\documentclass{amsart}
\usepackage{amsmath, amssymb, mathrsfs, verbatim, multirow}
\usepackage[all]{xy}
\usepackage{pifont}
\usepackage{float}

\newtheorem{Teo}{Theorem}[section]
\newtheorem{Prop}[Teo]{Proposition}
\newtheorem{Lema}[Teo]{Lemma}
\newtheorem{Cor}[Teo]{Corollary}

\theoremstyle{definition}

\newtheorem{Obs}[Teo]{Remark}
\newtheorem{Que}[Teo]{Question}
\newtheorem{Exa}[Teo]{Example}

\newtheorem{Con}[Teo]{Conjecture}

\newcommand{\Q}{\mathbb{Q}}

\newcommand{\Z}{\mathbb{Z}}
\newcommand{\N}{\mathbb{N}}

\newcommand{\Llr}{\Longleftrightarrow}
\newcommand{\lra}{\longrightarrow}
\newcommand{\Lra}{\Longrightarrow}
\newcommand{\VR}{\mathcal{O}}
\newcommand{\PI}{\mathfrak{p}}
\newcommand{\MI}{\mathfrak{m}}

\newcommand{\QF}{\mbox{\rm Quot}}

\begin{document}
\title[Essentially generation and invariants]{Essentially finite generation of valuation rings in terms of classical invariants}
\author{Steven Dale Cutkosky}
\thanks{Steven Dale Cutkosky was partially supported by NSF grant DMS-1700046.}

\author{Josnei Novacoski}
\thanks{During the realization of this project Josnei Novacoski was supported by a grant from Funda\c c\~ao de Amparo \`a Pesquisa do Estado de S\~ao Paulo (process number 2017/17835-9). Also, this work is partially result of an academic visit of Novacoski to the University of Missouri which was supported by a Miller Fellowship.}

\begin{abstract}
The main goal of this paper is to study some properties of an extension of valuations from classical invariants. More specifically, we consider a valued field $(K,\nu)$ and an extension $\omega$ of $\nu$ to a finite extension $L$ of $K$. Then we study when the valuation ring of $\omega$ is essentially finitely generated over the valuation ring of $\nu$. We present a necessary condition in terms of classic invariants of the extension  by Hagen Knaf and show that in some particular cases, this condition is also sufficient. We also study when the corresponding extension of graded algebras is finitely generated. For this problem we present an equivalent condition (which is weaker than the one for the finite generation of the valuation rings).
\end{abstract}


\maketitle
\section{Introduction}
Let $(K,\nu)$ be a valued field, $L$ a finite extension of $K$ and $\omega$ an extension of $\nu$ to $L$. We denote the valuation rings of $\omega$ and $\nu$ by $\VR_\omega$ and $\VR_\nu$, respectively. We also define the graded algebras associated to $\omega$ and $\nu$ by
\[
\textrm{gr}_\omega(\VR_\omega)=\bigoplus_{\gamma\in\Gamma_\omega}\{x\in \VR_\omega\mid \omega(x)\geq \gamma\}/\{x\in \VR_\omega\mid \omega(x)> \gamma\}
\]
and
\[
\textrm{gr}_\nu(\VR_\nu)=\bigoplus_{\gamma\in\Gamma_\nu}\{x\in \VR_\nu\mid \nu(x)\geq \gamma\}/\{x\in \VR_\nu\mid \nu(x)> \gamma\},
\]
respectively, where $\Gamma_\omega$ and $\Gamma_\nu$ are the respective value groups of $\omega$ and $\nu$. We have a natural inclusion of graded domains $\textrm{gr}_\nu(\VR_\nu)\rightarrow \textrm{gr}_\omega(\VR_\omega)$ which is an integral extension, so that $\textrm{gr}_\omega(\VR_\omega)$ is a finitely generated $\textrm{gr}_\nu(\VR_\nu)$-algebra if and and only if $\textrm{gr}_\omega(\VR_\omega)$ is a finitely generated $\textrm{gr}_\nu(\VR_\nu)$-module.

A ring extension $A \subseteq B$ is called essentially finitely generated  if there exist $x_1,\ldots,x_r\in B$ such that
\[
B = S^{-1}A[x_1 ,\ldots, x_r]
\]
for some multiplicative set $S \subseteq A[x_1,\ldots,x_r]$. We also say that $B$ is essentially of finite type over $A$ if $B$ is essentially finitely generated over $A$.
This paper is devoted, mainly to  the following question:
\begin{Que}\label{mainque}
\begin{description}
\item[(i)] When is $\VR_\omega$ essentially finitely generated over $\VR_\nu$?
\item[(ii)] When is $\textrm{gr}_\omega(\VR_\omega)$ a finitely generated $\textrm{gr}_\nu(\VR_\nu)$-algebra?
\end{description}
\end{Que}

The main application that we have in mind is to the  problem of local uniformization in positive characteristic. One program to solve this problem is by using \textit{ramification theory}. In this direction, Knaf and Kuhlmann proved that every \textit{Abhyankar valuation} admits local uniformization (see \cite{KK}) and that every valuation admits local uniformization in a suitable finite separable extension of the function field (see \cite{KK_1}). In order to prove local uniformization for the extension $(L|k,\nu)$ they proved that, in each situation, one can find another field $K$, $k\subseteq K\subseteq L$ such that $(K|k,\nu)$ admits local uniformization and that $L$ is contained in the \textit{absolute inertia field of $K$}. In the process, they use that if $L$ is contained in the absolute inertia field of $K$, then the extension of the valuation rings is essentially finitely generated. Moreover, the property of having valuation rings essentially finitely generated seems to be necessary if one wants to lift local uniformization.

Another reason to study such problems is to understand the structure of valuation rings. Valuation rings, despite being almost never noetherian, seem to share important properties with some classes of noetherian local rings. For example, in \cite{Dat} it is proved that the Frobenius map is always flat for valuation rings. This work is motivated by a famous theorem of Kunz (see \cite{Kun}) that states that a noetherian local ring is regular if and only if the Frobenius map is flat.

It is known that if $R$ is an  excellent local domain with quotient field $K$, $L$ is a finite extension of $K$ and $D$ is the integral closure of $R$ in $L$, then $D$ is a finitely generated  R-module \cite[Scholie IV.7.8.3 (vi)]{EGAIV}. We are led to ask whether the same property holds for a valuation ring, i.e., if $\VR$ is a valuation ring with quotient field $K$ and $L$ a finite extension of $K$, then is the integral closure of $\VR$ in $F$ a finitely generated $\VR$-algebra?  If that is the case, then every valuation ring of $L$ extending $\nu$ is essentially finitely generated over $\VR_\nu$ (because all such valuation rings are of the form $D_\PI$ where $\PI$ is a prime ideal of the integral closure $D$ of $\VR$ in $L$). 

Question \ref{mainque} \textbf{(i)} has a positive answer in the case where $L$ lies in the absolute inertia field of $K$ (see \cite{Hens}). In \cite{Hens}, it was studied also, in the case where $L$ lies in the absolute inertia field of $K$, under which conditions the extension of valuation rings is finitely generated (i.e., when we do not need to localize at a prime ideal). In that paper it was also suggested that items \textbf{(i)} and \textbf{(ii)} of Question \ref{mainque} are closely related. In this paper, we will give characterizations of when this statements are true in terms of classic invariants of extensions of valuation rings.

Throughout this paper we will fix the following notations and assumptions:
\begin{equation}                           \label{sit}
\left\{\begin{array}{ll}
L|K & \mbox{is a finite field extension,}\\
\omega & \mbox{is a valuation on } L\\
\nu & \mbox{is the restriction of }\omega\mbox{ to }K,\\
\VR_\omega\mbox{ and }\VR_\nu &\mbox{are the valuation rings of }\omega\mbox{ and }\nu\mbox{, respectively,}\\
\Gamma_\omega\mbox{ and }\Gamma_\nu &\mbox{are the value groups of }\omega\mbox{ and }\nu\mbox{, respectively,}\\
F\omega\mbox{ and }K\nu &\mbox{are the residue fields of }\omega\mbox{ and }\nu\mbox{, respectively.}
\end{array}\right.
\end{equation}

The \textbf{ramification} and \textbf{inertia} indexes of the extension $\omega|\nu$  are defined as
\[
e(\omega|\nu)=(\Gamma_\omega:\Gamma_\nu)\mbox{ and }f(\omega|\nu)=[F\omega:K\omega],
\]
respectively. We also define the henselization of $(K,\nu)$ (for a fixed extension $\overline \nu$ of $\nu$ to an algebraic closure $\overline K$ of $K$) as the smallest field $K^h$ such that $K\subseteq K^h\subseteq \overline K$ and that $\overline \nu$ is the only extension of $\overline \nu|_{K^h}$ to $\overline K$ (\cite[Section 17]{End}). One can prove (see Corollary 5.3.8 and discussion on page 75 of \cite{resteengler}) that
\[
d(\omega|\nu)=\frac{\left[L^h:K^h\right]}{e(\omega|\nu)\cdot f(\omega|\nu)}
\]
is a positive integer which we will call the \textbf{defect} of $\omega|\nu$ (in fact, one can prove that $d(\omega|\nu)=1$ if ${\rm char}(K\nu)=0$ and $d(\omega|\nu)=p^n$ for some $n\in\N$ if ${\rm char}(K\nu)=p>0$).

For an extension $\Delta\subseteq\Gamma$ of ordered abelian groups, we define the \textbf{initial index} of $\Delta$ in $\Gamma$ \cite[page 138]{End} as
\[
\epsilon(\Gamma\mid\Delta):=|\{\gamma\in\Gamma_{\geq 0}\mid\gamma<\Delta_{>0}\}|,
\]
where
\[
\Gamma_{\geq 0}=\{\gamma\in \Gamma\mid\gamma\geq 0\}\mbox{ and }\Delta_{> 0}=\{\delta\in \Delta\mid\delta> 0\}.
\]
Hence, we can also define the initial index $\epsilon(\omega|\nu)$ of the extension $\omega|\nu$ as $\epsilon(\Gamma_\omega:\Gamma_\nu)$. 

If $S$ is a subsemigroup of an abelian semigroup  $T$, we will say that $T$ is a finitely generated $S$-module, or that $S$ has finite index in $T$ if there exist a finite number of elements $\gamma_1,\ldots,\gamma_t\in T$ such that 
$$
T=\bigcup_{i=1}^t(\gamma_i+S).
$$

We now consider the following statements:\\
\noindent\ding{182}. $\VR_\omega$ is essentially finitely generated over $\VR_\nu$. \\
\noindent\ding{183}. $\textrm{gr}_\omega(\VR_\omega)$ is finitely generated over $\textrm{gr}_\nu(\VR_\nu)$.\\
\noindent\ding{184}. $\left(\Gamma_{\nu}\right)_{\geq 0}$ has finite index in $\left(\Gamma_{\omega}\right)_{\geq 0}$.\\
\noindent\ding{185}. The integral closure $D$ of $\VR_\nu$ in $L$ is a finitely generated $\VR_\nu$-algebra.\\
\noindent\ding{186}. The integral closure $D$ of $\VR_\nu$ in $L$ is a finite module over $\VR_\nu$.\\
\noindent\ding{187}. $\epsilon(\omega|\nu)=e(\omega|\nu)$.\\
\noindent\ding{188}. $\epsilon(\omega|\nu)=e(\omega|\nu)$ and $d(\omega|\nu)=1$.\\
\noindent\ding{189}. $\epsilon(\omega_i|\nu)=e(\omega_i|\nu)$ and $d(\omega_i|\nu)=1$, where $\omega_1,\ldots,w_r$ are all the extensions of $\nu$ to $L$.

One of the main pourposes of this paper is to prove the implications in the following diagram:
\begin{displaymath}
\begin{array}{ccccc}
\mbox{\ding{185}}& \Llr & \mbox{\ding{186}}&\Llr &\mbox{\ding{189}}\\
\not\Uparrow\Downarrow && & &\not\Uparrow\Downarrow\\
\mbox{\ding{182}}&&\Longrightarrow &&\mbox{\ding{188}}  \\
 & & & &  \not\Uparrow\Downarrow \\
\mbox{\ding{183}}& \Llr& \mbox{\ding{184}}&\Llr & \mbox{\ding{187}}\\
\end{array}
\end{displaymath}

Observe that
\[
\mbox{\ding{189}}\Lra \mbox{\ding{188}}\Lra\mbox{\ding{187} and \ding{186}}\Lra\mbox{\ding{185}}
\]
are trivial and \ding{185} $\Lra$ \ding{182} follows from the fact that $\VR_\omega=D_{D\cap \MI\omega}$. Also, the counter-examples for \ding{182} $\not\Rightarrow$ \ding{185}, \ding{187} $\not\Rightarrow$ \ding{188} and \ding{188} $\not\Rightarrow$ \ding{189} are not difficult and will be presented in Section \ref{proofofdieamg}.

The equivalence \ding{189} $\Llr$ \ding{186} is Theorem 18.6 of \cite{End} and we will show \ding{185} $\Lra$ \ding{186} in Section \ref{proofofdieamg} (Proposition \ref{modvsalg}). The equivalences
\[
\mbox{\ding{183}} \Llr \mbox{\ding{184}}\Llr \mbox{\ding{187}}
\]
are the main subject of Section \ref{initialindex} and the implication \ding{182} $\Lra$ \ding{188} is the main subject of Section \ref{Necessity}.  We present in Section \ref{Necessity} a proof of this implication by Hagen Knaf.

It is proven in \cite{Hens} that if $L$ lies in the absolute inertia field of $(K,\nu)$ (this is equivalent to $e(\omega|\nu)=1$, $d(\omega|\nu)=1$ and that $L\omega|K\nu$ is separable), then $\VR_\omega$ is essentially finitely generated over $\VR_\nu$. Also, the condition $\epsilon(\omega|\nu)=e(\omega|\nu)$ and $d(\omega|\nu)=1$ implies that $\VR_\omega$ is essentially finitelly generated over $\VR_\nu$ in the cases where $L|K$ is normal and when $\omega$ is the only extension of $\nu$ to $L$ (see Corollary \ref{Endler1}). Hence, we are led to believe that this holds in general, i.e., that the following conjecture is true.

\begin{Con}[\ding{188} $\Lra$ \ding{182}]\label{MainConj}
Assume that we are in situation 
\begin{equation}                           \label{sit2}
\left\{\begin{array}{l}
\mbox{We are in situation (\ref{sit}); and}\\
d(\omega|\nu)=1 \mbox{ and }\epsilon(\omega\mid\nu)=e(\omega\mid\nu)
\end{array}\right.
\end{equation}
Then $\VR_\omega$ is essentially finitely generated over $\VR_\nu$.
\end{Con}
In Section \ref{Secredsep} we will show that in order to prove the conjecture above, we can assume that $L|K$ is separable, i.e., that it is enough to show the following conjecture.

\begin{Con}\label{Conjectsep}
Assume that we are in situation (\ref{sit2}) and that $L|K$ is separable. Then $\VR_\omega$ is essentially finitely generated over $\VR_\nu$.
\end{Con}

One of the main results of this paper is that Conjecture \ref{MainConj} is satisfied for the case where $\nu$ is centered in a quasi-excellent local domain of dimension two, or if 
$\nu$ is an Abhyankar valuation.

If $R$ is a local domain which is contained in $K$, then we say that $\nu$ is centered in $R$ is $R\subset \mathcal O_\nu$ and $m_\nu\cap R=m_R$.

\begin{Teo}\label{Teocutko}
Assume that $R$ is an excellent two-dimensional local domain with quotient field $K$. Suppose that $\nu$ is a valuation of $K$ centered at $R$. Assume that $L$ is a finite separable extension of $K$ and that $\omega$ is an extension of $\nu$ to $L$. If $d(\omega\mid\nu)=1$ and $\epsilon(\omega\mid\nu)=e(\omega\mid\nu)$, then $\VR_\omega$ is essentially finitely generated over $\VR_\nu$.
\end{Teo}
Section \ref{Sectproofmanotehm} is devoted to the proof of  Theorem \ref{Teocutko}.

Let $K$ be an algebraic function field over a field $k$. A valuation $\nu$ on $K$ (which is trivial on $k$) satisfies the fundamental inequality (\cite[Lemma 1]{Ab})
$$
{\rm trdeg}_kK\ge {\rm trdeg}_k\VR_\nu/\mathfrak m_\nu+{\rm ratrk}(\nu).
$$
Here the rational rank ${\rm ratrk}(\nu)$ of $\nu$ is the $\Q$-dimension of the tensor product of the value group $\Gamma_{\nu}$ of $\nu$ with $\Q$. We say that $\nu$ is an Abhyankar valuation if equality holds in this equation. 

\begin{Teo}\label{TeoAbh} Let $K$ be an algebraic function field over a field $k$, and let $\nu$ be an Abhyankar valuation on $K$.  
Assume that $L$ is a finite  extension of $K$ and that $\omega$ is an extension of $\nu$ to $L$. Suppose that $\VR_\omega/\mathfrak m_\omega$ is separable over $k$. If $d(\omega|\nu)=1$ and $\epsilon(\omega|\nu)=e(\omega|\nu)$, then $\VR_\omega$ is essentially finitely generated over $\VR_\nu$.
\end{Teo}

We prove  Theorem \ref{TeoAbh} in Section \ref{SecAbh}.

\noindent
\bf Acknowledgement. \rm We are grateful to Hagen Knaf for suggesting the main problem of this paper and for interesting discussions on the topic. Also, he contributed substantially to this paper by providing the proof for Theorem \ref{Knaf} presented here.

\section{Preliminaries}\label{proofofdieamg}

We will start by stating some known results related to this paper. For a subring $R$ of $\VR_\nu$ (and $S$ of $\VR_\omega$) we will denote $R_\nu=R_{\MI_\nu\cap R}$ ($S_\omega=S_{\MI_\omega\cap S}$).

\begin{Teo}(\cite[Theorem 18.6]{End})\label{Endler}
Assume that we are in situation (\ref{sit}) and let $\omega=\omega_1,\ldots,\omega_r$ be all the extensions of $\nu$ to $L$. Let $D$ be the integral closure of $\VR_\nu$ in $L$. Then $D$ is a finite free module over $\VR_\nu$ if and only if $(L|K,\nu)$ is defectless and $\epsilon(\omega_i|\nu)=e(\omega_i|\nu)$ for every $i=1,\ldots,r$.
\end{Teo}
\begin{Cor}\label{Endler1}
Assume that we are in one of the following case:\\
\textbf{(i)} $\omega$ is the only extension of $\nu$ to $L$;\\
\textbf{(ii)} $L|K$ is a normal extension.\\
Then $D$ is a finite $\VR_\nu$-module if and only if $d(\omega|\nu)=1$ and $\epsilon(\omega|\nu)=e(\omega|\nu)$.
\end{Cor}
\begin{proof}
Item \textbf{(i)} is trivial and item \textbf{(ii)} follows from the fact that if $L|K$ is normal, any every extension $\omega_i$ of $\nu$ to $L$ is conjugate to $\omega$ and hence $d(\omega_i|\nu)=d(\omega|\nu)$, $\epsilon(\omega_i|\nu)=\epsilon(\omega|\nu)$ and $e(\omega_i|\nu)=e(\omega|\nu)$.
\end{proof}
\begin{Teo}(\cite[Theorem 1.3]{Hens})\label{teoessnovkuhl}
Assume that we are in situation (\ref{sit}). If $e(\omega|\nu)=d(\omega|\nu)=1$ and $L\omega|K\nu$ is separable, then
\[
\VR_\omega=\VR_\nu[\eta]_\omega\mbox{ for some }\eta\in\VR_\omega.
\]
In particular, $\VR_\omega$ is essentially finitely generated over $\VR_\nu$.
\end{Teo}

We will present now the examples that show that the ``up arrows" in our diagram are satisfied.

\begin{Exa}(\ding{187} $\nRightarrow$ \ding{188})
For this, it is enough to present an immediate extension with non trivial defect. Example 3.1 of \cite{Kuhldefect} is a simple example of that situation.
\end{Exa}

\begin{Exa}(\ding{182} $\nRightarrow$ \ding{185} and \ding{188} $\nRightarrow$ \ding{189})
Let $(K,\nu)$ be a valued field with group of values $\frac{1}{3^{\infty}}\Z$. Let $L$ be the extension
\[
K(\eta)=K[X]/(X(X-1)^2-a)
\]
where $a\in K$ is such that $\nu(a)=1$. Then $\nu$ admits two extensions $\omega_1$ and $\omega_2$ to $L$ defined by
\[
\omega_1(\eta)=1\mbox{ and }\omega_2(\eta-1)=\frac 12.
\]
Since
\[
e(\omega_1|\nu)\ge1\mbox{ and }e(\omega_2|\nu)\ge2,
\]
we have by the fundamental inequality, that these are the only extensions of $\nu$ to $L$. Moreover, we have that
\[
e(\omega_1|\nu)=1, e(\omega_2|\nu)=2\mbox{ and }d(\omega_i|\nu)=f(\omega_i|\nu)=1\mbox{ for }i=1,2.
\]
Since the value group of $\nu$ is not discrete, we have that $\epsilon(\omega_i|\nu)=1$ for $i=1,2$. Hence,  the extension $\omega_1$ satisfies \ding{188} but not \ding{189}, since $\epsilon(\omega_2|\nu)<e(\omega_2|\nu)$.

Moreover, since $\epsilon(\omega_2|\nu)<e(\omega_2|\nu)$, by Theorem \ref{Endler} we conclude that \ding{185} is not satisfied. By Theorem \ref{teoessnovkuhl}, we have that \ding{182} is satisfied. In fact, one can show that
\[
\VR_{\omega_1}=\VR_\nu[\eta]_{\omega_1}.
\]

\end{Exa}
\begin{Prop}[\ding{185} $\Lra$ \ding{186}]\label{modvsalg}
Let $A$ be a normal domain (integrally closed in $K=\QF(A)$) and $L$ an algebraic extension extension of $K$. Let $I_L(A)$ be the integral closure of $A$ in $L$. If $I_L(A)$ is a finitely generated $A$-algebra, then it is a finite $A$-module.
\end{Prop}
\begin{proof}
Take any $x\in I_L(A)$. Since $x$ is integral over $A$, there exists a relation
\[
x^{n}+a_{n-1}x^{n-1}+\ldots+a_0=0,\mbox{ for }a_0,\ldots,a_{n-1}\in A.
\]
This means that
\[
x^{n}=-a_{n-1}x^{n-1}-\ldots-a_0\in A_x:=Ax^{n-1}+\ldots+Ax+A.
\]
Assume now that for $k\geq n$, we have $x^i\in A_x$ for every $1\leq i\leq k$. Then
\begin{displaymath}
\begin{array}{rcl}
x^{k+1}&=& xx^k=x(b_{n-1}x^{n-1}+b_{n-2}x^{n-2}\ldots+b_0)\\
&=&b_{n-1}x^{n}+b_{n-2}a^{n-1}\ldots+b_0x\\
&=&(b_{n-2}-b_{n-1}a_{n-1})x_i^{n-1}+\ldots+(b_0-b_{n-1}a_1)x-b_{n-1}a_0\in A_x.
\end{array}
\end{displaymath}
Therefore, $A[x]=A_x$.

If $I_L(A)=A[x_1,\ldots,x_r]$, then we choose $n_1,\ldots,n_r$ such that $x_i^{n_i}\in A_{x_i}$. Then we can prove, by induction on $s$, $1\leq s\leq r$, that $A[x_1,\ldots,x_s]$ is generated as an $A$-module by $x_1^{j_1}\ldots x_s^{j_s}$ where $0\leq j_i<n_i$ and $1\leq i\leq s$. Therefore, $I_L(A)$ is a finite $A$-module.
\end{proof}

\section{The initial index}\label{initialindex}
In this section we prove a few basic results about the initial index of a subgroup $\Delta$ of finite index in an ordered group $\Gamma$.

\begin{Lema}\label{Lemma1CS}
Let $\Gamma_0$ be the first convex subgroup of $\Gamma$. If $1<\epsilon(\Gamma\mid\Delta)<\infty$, then $\Gamma_0\simeq\Z$.
\end{Lema}

\begin{proof}
Assume, towards a contradiction, that $\Gamma_0\not\simeq \Z$. Then $\Gamma$ does not admit a smallest positive element. Since $\epsilon>1$, there exists $\gamma\in\Gamma_{>0}$ such that $\gamma<\Delta_{>0}$. Since $\Gamma$ does not admit a smallest positive element, there would exist infinitely many positive elements in $\Gamma$ smaller than $\Delta_{>0}$, and this is a contradiction to $\epsilon<\infty$.
\end{proof}
\begin{Prop}\label{lemmagammakk1}
Assume that $\epsilon:=\epsilon(\Gamma\mid\Delta)>1$ and denote
\[
0=\gamma_0<\gamma_1<\ldots<\gamma_{\epsilon-1}
\]
for all the elements in $\Gamma_{\geq 0}$ which are smaller than $\Delta_{>0}$. Then
\begin{equation}\label{eqannextgamma}
\gamma_k=k\gamma_1\mbox{ for every }k,\ 1\leq k\leq \epsilon-1\mbox{ and }\epsilon\gamma_1\in\Delta.
\end{equation}
\end{Prop}
\begin{proof}
We will prove it by induction on $k$. For $k=1$, the first assertion of (\ref{eqannextgamma}) is trivially satisfied. Assume that for a given $k$, $1\leq k\leq \epsilon-2$, we have $\gamma_k=k\gamma_1$. We will show that $\gamma_{k+1}=(k+1)\gamma_1$. Since $\gamma_k=k\gamma_1$, we have
\[
(k+1)\gamma_1=k\gamma_1+\gamma_1=\gamma_k+\gamma_1,
\]
hence we have to show that $\gamma_{k+1}=\gamma_k+\gamma_1$.

Since $\gamma_k<\gamma_{k+1}$ we have $0< \gamma_{k+1}-\gamma_k$ and because $\gamma_1$ is the smallest positive element of $\Gamma$, we obtain that $\gamma_1\leq\gamma_{k+1}-\gamma_k$. Hence, $\gamma_k+\gamma_1\leq \gamma_{k+1}$. Since $\gamma_k+\gamma_1\leq \gamma_{k+1}<\Delta_{>0}$, there exists $j$, $1\leq j\leq \epsilon-1$, such that $\gamma_k+\gamma_1=\gamma_j$. On the other hand, since $\gamma_k<\gamma_j$, by our assumption on $\gamma_i$'s, $\gamma_{k+1}\leq\gamma_j=\gamma_k+\gamma_1$. Therefore, $\gamma_{k+1}=\gamma_k+\gamma_1$, which is what we wanted to prove.

Since $\gamma_{\epsilon-1}$ is the largest element in $\Gamma_{\geq 0}$ which is smaller than every element $\delta\in\Delta_{>0}$, there exists $\delta\in \Delta$ such that
\[
\gamma_{\epsilon-1}<\delta\leq\gamma_{\epsilon-1}+\gamma_1=(\epsilon-1)\gamma_1+\gamma_1=\epsilon\gamma_1.
\]
If $\delta<\epsilon\gamma_1$, we would have $0<\delta-\gamma_{\epsilon-1}<\gamma_1$, which is a contradiction. Therefore, $\delta=\epsilon\gamma_1$, which completes our proof.
\end{proof}

\begin{Prop}[\ding{184} $\Llr$ \ding{187}]\label{Propoequ28}
Let $\Gamma$ be an ordered abelian group and $\Delta$ a subgroup of $\Gamma$ of finite index. We have:
\begin{description}
\item[(i)] $\epsilon(\Gamma\mid\Delta)\leq (\Gamma:\Delta)$; and
\item[(ii)] $\epsilon(\Gamma\mid\Delta)=(\Gamma:\Delta)$ if and only if $\Delta_{\geq 0}$ has finite index in $\Gamma_{\geq 0}$, i.e., if there exist $\gamma_0,\ldots,\gamma_{n-1}\in \Gamma_{\geq 0}$ ($\gamma_0=0$) such that
\[
\Gamma_{\geq 0}=\bigcup_{i=0}^{n-1}\left(\gamma_i+\Delta_{\geq 0}\right).
\]
\end{description} 
\end{Prop}
\begin{proof}
Set $\epsilon:=\epsilon(\Gamma\mid\Delta)$ and choose $\gamma_0,\ldots,\gamma_{\epsilon-1}\in\Gamma_{\geq 0}$ such that
\[
0=\gamma_0<\ldots<\gamma_{\epsilon-1}<\Delta_{>0}.
\]
We claim that the cosets of the $\gamma_i$'s modulo $\Delta$ are disjoint, i.e, that
\[
\left(\gamma_i+\Delta\right)\cap \left(\gamma_j+\Delta\right)=\emptyset\mbox{ if }i\neq j.
\]
Assume otherwise. Then, there would exist $i$ and $j$, $0\leq i<j\leq \epsilon-1$ such that
\[
\gamma_j-\gamma_i=\delta\in\Delta.
\]
Then
\[
0<\gamma_j-\gamma_i=\delta<\gamma_j,
\]
which is a contradiction to our assumption on the $\gamma_i$'s.

Since the cosets of the $\gamma_i$'s are disjoint, we have \textbf{(i)}. To prove \textbf{(ii)} we assume first that $\epsilon=(\Gamma:\Delta)$. This, together with the fact that the $\gamma_i$'s have distinct cosets modulo $\Delta$ gives us that
\[
\Gamma=\bigcup_{i=0}^{\epsilon-1}\left(\gamma_i+\Delta\right).
\]

Now take any element $\gamma\in\left(\Gamma\right)_{\geq 0}$. Since $\gamma\in\Gamma$, there exist $i$, $0\leq i\leq \epsilon-1$, and $\delta\in\Delta$ such that $\gamma=\gamma_i+\delta$. We claim that $\delta\geq 0$. Indeed, if $\delta<0$, then
\[
0\leq\gamma=\gamma_i+\delta<\gamma_i,
\]
and consequently $\gamma=\gamma_j$ for some $j$. This implies that
\[
\left(\gamma_i+\Delta\right)\cap \left(\gamma_j+\Delta\right)\neq\emptyset\mbox{ and }\gamma_i\neq\gamma_i,
\]
and this is a contradiction to what we proved in the previous paragraph. Therefore,
\[
\left(\Gamma\right)_{\geq 0}=\bigcup_{i=0}^{\epsilon-1}\left(\gamma_i+\left(\Delta\right)_{\geq 0}\right).
\]

Assume now that
\begin{equation}\label{equatgrousme}
\left(\Gamma\right)_{\geq 0}=\bigcup_{i=0}^{n-1}\left(\gamma_i+\left(\Delta\right)_{\geq 0}\right),
\end{equation}
for some $\gamma_0,\ldots,\gamma_{n-1}\in \Gamma_{\geq 0}$. We claim that we can choose the $\gamma_i$'s such that
\[
\left(\gamma_i+\Delta_{\geq 0}\right)\cap \left(\gamma_j+\Delta_{\geq 0}\right)= \emptyset,\mbox{ for }0\leq i< j\leq n-1.
\]
Indeed, if for some $i$ and $j$, $0\leq i,j\leq n-1$, we have
\[
\left(\gamma_i+\Delta_{\geq 0}\right)\cap \left(\gamma_j+\Delta_{\geq 0}\right)\neq \emptyset,
\]
then there exist $\delta_i,\delta_j\in\Delta_{\geq 0}$ such that $\gamma_i+\delta_i=\gamma_j+\delta_j$. Assume, without loss of generality, that $\gamma_j>\gamma_i$. Then $\delta_i-\delta_j=\gamma_j-\gamma_i\in \Delta_{\geq 0}$ and consequently
\[
\gamma_j=\gamma_i+\delta_i-\delta_j\in\gamma_i+\Delta_{\geq 0}.
\]
Hence, $\gamma_j+\Delta_{\geq 0}\subseteq \gamma_i+\Delta_{\geq 0}$ and we can remove $\gamma_j+\Delta_{\geq 0}$ in (\ref{equatgrousme}). Since there are only finitely many $\gamma_i$'s, we proceed as above until we reach disjoint cosets modulo $\Delta_{\geq 0}$. Assume, without loss of generality, that
\[
0=\gamma_0<\gamma_1<\ldots<\gamma_{n-1}.
\]
We will show that
\[
\Gamma=\bigcup_{i=0}^{n-1}\left(\gamma_i+\Delta\right),
\]
that $\left(\gamma_i+\Delta\right)\cap \left(\gamma_j+\Delta\right)=\emptyset$ when $i\neq j$ and that $\gamma_{n-1}<\delta$ for every $\delta\in\Delta_{>0}$. This will imply that $(\Gamma:\Delta)=n\leq \epsilon$, and since $\epsilon\leq (\Gamma:\Delta)$ we conclude that $\epsilon=(\Gamma:\Delta)$.

Take $\gamma\in\Gamma$. If $\gamma\geq 0$, then by our assumption
\[
\gamma\in\gamma_i+\Delta_{\geq 0}\subseteq \gamma_i+\Delta\mbox{ for some }i, 0\leq i\leq n-1
\]
and if $\gamma\in\Delta$, then
\[
\gamma\in \gamma_0+\Delta,
\]
because $\gamma_0=0$. Assume now that $\gamma<0$ and that $\gamma\in\Gamma\setminus\Delta$. Since $\Delta$ has finite index in $\Gamma$, there exist $\delta\in\Delta$ and $r\in\N$, $r>1$, such that $r\gamma=\delta$. Since $\gamma<0$, we have $\delta=r\gamma<\gamma$ and hence $\gamma-\delta\in\Gamma_{\geq 0}$. By our assumption, there exists $\gamma_i$ and $\delta'\in\Delta_{\geq 0}$ such that $\gamma-\delta=\gamma_i+\delta'$ and consequently $\gamma\in\gamma_i+\Delta$. Therefore,
\begin{equation}\label{eqgroupsposets}
\Gamma=\bigcup_{i=0}^{n-1}\left(\gamma_i+\Delta\right).
\end{equation}

Assume now that $\left(\gamma_i+\Delta\right)\cap\left(\gamma_j+\Delta\right)\neq \emptyset$ for some $i\leq j$. This means that there exist $\delta_i,\delta_j\in\Delta$ such that $\gamma_i+\delta_i=\gamma_j+\delta_j$. Since $i\leq j$, this means that $\delta_i-\delta_j=\gamma_j-\gamma_i\in \Delta_{\geq 0}$ and since we assumed that the cosets in (\ref{equatgrousme}) are disjoint, we must have that $i=j$. This implies that the cosets in (\ref{eqgroupsposets}) are disjoint and therefore $n=(\Gamma:\Delta)$.

It remains to show that $n\leq \epsilon$, i.e., that $\gamma_{n-1}<\delta$ for every $\delta\in\Delta_{>0}$. If this were not the case, then there would exist $\delta\in\Delta$ such that $0 < \delta\leq \gamma_{n-1}$ and consequently $\gamma_{n-1}-\delta\in\Gamma_{\geq 0}$. Since $\delta>0$ we have that $\gamma_{n-1}-\delta<\gamma_{n-1}$ and hence there exists $i$, $0\leq i<n-1$ such that $\gamma_{n-1}-\delta=\gamma_i+\delta'$. Therefore, $\gamma_{n-1}-\gamma_i\in\Delta$, which is a contradiction to the fact that the cosets in (\ref{eqgroupsposets}) are disjoint. This concludes our proof.
\end{proof}
\begin{Obs}\label{RemCS}
\textbf{(i)} We observe from the proof of the above proposition, that if $\epsilon:=\epsilon(\Gamma:\Delta)=(\Gamma:\Delta)$, then
\[
\Gamma=\bigcup_{i=0}^{\epsilon-1}\left(\gamma_i+\Delta\right)\mbox{ and }\Gamma_{\geq 0}=\bigcup_{i=0}^{\epsilon-1}\left(\gamma_i+\Delta_{\geq 0}\right),
\]
where $0=\gamma_0<\gamma_1<\ldots<\gamma_{\epsilon-1}$ are all the elements in $\Gamma_{\geq 0}$ smaller than $\Delta_{>0}$.\\
\textbf{(ii)} From the above propositions we conclude that if $\epsilon(\Gamma|\Delta)=(\Gamma:\Delta)$, then
\[
\Gamma/\Delta\simeq \Z/\epsilon \Z.
\]
\end{Obs}
\begin{Cor}\label{sgroupcor}
Assume that $\Gamma=\Gamma_1\oplus_{\rm lex}\Z$. If $\epsilon:=\epsilon(\Gamma\mid\Delta)=(\Gamma:\Delta)>1$, then for every $\gamma_1\in \left(\Gamma_1\right)_{>0}$ there exists $a\in\Z$ such that $(\gamma_1,a)\in \Delta$.
\end{Cor}
\begin{proof}
Since $(0,1)$ is the smallest element of $\Gamma_{> 0}$, Proposition \ref{lemmagammakk1} gives us that $(0,i)\notin\Delta$ for every $i$, $1\leq i<\epsilon$, and $(1,\epsilon)\in\Delta$. Since $\epsilon(\Gamma|\Delta)=(\Gamma:\Delta)$, Proposition \ref{Propoequ28} (and the remark above) gives us
\[
\Gamma_{\geq 0}=\bigcup_{i=0}^{\epsilon-1} (0,i)+\Delta_{\geq 0}.
\]
Since $(\gamma_1,0)\in \Gamma_{\geq 0}$, there exists $(b,c)$ $\in\Delta_{\geq 0}$ and $i$, $0\leq i\leq \epsilon-1$ such that $(\gamma_1,0)=(b,c)+(0,i)$.  Hence, $(\gamma_1,-i)=(b,c)\in\Delta_{\geq 0}$.
\end{proof}

As a consequence of the previous corollary, if $\Gamma=\Z^n$ (with the lexicographic ordering) and $\Delta\subseteq \Gamma$ is such that $\epsilon(\Gamma|\Delta)=(\Gamma:\Delta)$, then there exist $a_{i,j}\in\Z$, for $1\leq i\leq n-1$ and $i+1\leq j\le n$, such that
\[
(1,a_{12},\ldots,a_{1n}),\ldots, (0,\ldots,0,1,a_{n-1n})\in\Delta.
\]
Indeed, setting $\Gamma_1=\Z^{n-1}$, the previous corollary says that for every $\gamma_1\in \Z^{n-1}$, $\gamma_1>0$, there exists $a\in \Z$ such that $(\gamma_1,a)\in \Delta$. Hence, if we take
\[
\gamma_1=(1,a_{12},\ldots,a_{1n-1}),\ldots, \gamma_{n-1}=(0,\ldots,0,1)\in \Gamma_1
\]
(we could choose $a_{ij}=0$ if we wanted), there exist $a_{1n},\ldots, a_{n-1n}\in \Z$ such that
\[
(1,a_{12},\ldots,a_{1n}),\ldots, (0,\ldots,0,1,a_{n-1n})\in\Delta.
\]

We will now prove the converse of the previous statement.
\begin{Prop}
Assume that $\Gamma=\Z^n$ (with lexicographic ordering) and that $\Delta$ is a subgroup of $\Gamma$ of finite order. If there exist $a_{i,j}\in\Z$ with $1\le i\le n-1$ and $i+1\le j\le n$ such that
\[
(1,a_{12},\ldots,a_{1n}),\ldots, (0,\ldots,0,1,a_{n-1n})\in\Delta,
\]
then $\epsilon(\Gamma:\Delta)=(\Gamma:\Delta)$.
\end{Prop}
\begin{proof}
First we observe that if there exist $a_{i,j}\in\Z$, $1< i\leq j\leq n$ such that
\[
(1,a_{12},\ldots,a_{1n}),\ldots, (0,\ldots,0,1,a_{n-1n})\in\Delta,
\]
then (using ``Gaussian elimination") there exist $a_1,\ldots, a_{n-1}\in\Z$ such that
\[
\alpha_1:=(1,0,\ldots,0,a_1),\ldots,\alpha_{n-1}:= (0,\ldots,0,1,a_{n-1})\in\Delta.
\]
If $\epsilon(\Gamma|\Delta)=1$, then $(0,\ldots,0,1)\in\Delta$ and we can use again Gaussian elimination to obtain that $(1,\ldots,0,0),\ldots, (0,\ldots,0,1)\in\Delta$ and hence $\Gamma=\Delta$.

If $\epsilon(\Gamma:\Delta)>1$, then by Lemma \ref{lemmagammakk1}, we have that $(0,\ldots,0,1),\ldots,(0,\ldots,0,\epsilon-1)\notin\Delta$ and $(0,\ldots,0,\epsilon)\in\Delta$. We will show that
\[
\Gamma_{\geq 0}=\bigcup_{i=0}^{\epsilon-1}(0,\ldots,0,i)+\Delta_{\geq 0}
\]
and the result will follow from Proposition \ref{Propoequ28} \textbf{(ii)}. Take any $\gamma=(b_1,\ldots,b_n)\in\Gamma_{\geq 0}$. If $b_1=b_2=\ldots=b_{n-1}=0$, then $b_n\geq 0$ and there exist $c'\in\N$ and $i$, $0\leq i\leq \epsilon-1$, such that $b_n=i+c'\epsilon$. Hence, $\gamma=(0,\ldots,0,i)+\delta$ where $\delta=c'(0,\ldots,0,\epsilon)\in\Delta_{\geq 0}$. If $b_j\neq 0$ for some $j$, $1\leq j\leq n-1$, then there exists $i$, $1\leq i\leq n-1$, such that $b_j=0$ for $j<i$ and $b_i>0$. Then, we can write
\[
\gamma=b_i\alpha_i+b_{i+1}\alpha_{i+1}+\ldots+b_{n-1}\alpha_{n-1}+(0,\ldots,0,c)
\]
for some $c\in\Z$. Then there exists $c'\in\Z$ and $i$, $1\leq i\leq {\epsilon-1}$, such that
\[
c=i+c'\epsilon.
\]
Therefore, $\gamma=(0,\ldots,0,i)+\delta$, where
\[
\delta=b_i\alpha_i+b_{i+1}\alpha_{i+1}+\ldots+b_{n-1}\alpha_{n-1}+c'(0,\ldots,0,\epsilon)\in\Delta_{\geq 0}.
\]
This concludes our proof.
\end{proof}
The next lemma shows that $\epsilon$ is multiplicative.
\begin{Lema}\label{epsilonmiult}
Let $\Delta\subseteq\Sigma\subseteq\Gamma$ be ordered abelian groups (with $\Delta\subseteq\Sigma$, $\Sigma\subseteq\Gamma$ and $\Delta\subseteq\Gamma$ of finite index). Then
\[
\epsilon(\Gamma\mid\Sigma)\cdot\epsilon(\Sigma\mid\Delta)= \epsilon(\Gamma\mid\Delta).
\]
In particular, $\epsilon(\Gamma\mid\Delta)=(\Gamma:\Delta)$ if and only if
\[
\epsilon(\Gamma\mid\Sigma)=(\Gamma:\Sigma)\mbox{ and }\epsilon(\Sigma\mid\Delta)=(\Sigma:\Delta).
\]
\end{Lema}
\begin{proof}
Let $r=\epsilon(\Sigma\mid\Delta)$ and $s=\epsilon(\Gamma\mid\Sigma)$ and choose elements $\sigma_0,\ldots,\sigma_{r-1}\in\Sigma_{\geq 0}$ and $\gamma_0,\ldots,\gamma_{s-1}\in\Gamma_{\geq 0}$ such that
\[
0=\sigma_0<\ldots<\sigma_{r-1}<\Delta_{> 0}\mbox{ and }0=\gamma_0<\ldots<\gamma_{s-1}<\Sigma_{> 0}.
\]

If $\delta\in\Delta_{>0}$, then $\sigma_{r-1}<\delta$ and since $\Delta\subseteq\Sigma$ we have $\delta-\sigma_{r-1}\in\Sigma_{>0}$. Hence, $\gamma_{s-1}<\delta-\sigma_{r-1}$ which implies that
\[
\sigma_i+\gamma_j\leq \sigma_{r-1}+\gamma_{s-1}<\delta,\mbox{ for every }i\mbox{ and }j, 0\leq i\leq r-1\mbox{ and }0\leq j\leq s-1.
\]
Therefore,
\[
\epsilon(\Gamma\mid\Sigma)\cdot\epsilon(\Sigma\mid\Delta)=r\cdot s\leq \epsilon(\Gamma\mid\Delta).
\]

Take now $\gamma\in\Gamma_{\geq 0}$ such that $\gamma<\Delta_{>0}$ and let $j$, $0\leq j\leq r-1$ be the largest index for which $\sigma_j\leq \gamma$, i.e.,
\[
j=r-1\mbox{ and }\sigma_{r-1}\leq \gamma\mbox{ or }j<r-1\mbox{ and }\sigma_j\leq \gamma<\sigma_{j+1}.
\]
We claim that $\gamma-\sigma_j<\Sigma_{>0}$ and consequently $\gamma-\sigma_j=\gamma_i$ (i.e., $\gamma=\gamma_i+\sigma_j$) for some $i$, $0\leq i\leq s-1$. Indeed, if this were not the case, then there would exist $\sigma\in\Sigma_{>0}$ such that
\[
0<\sigma\leq\gamma-\sigma_j\leq\gamma<\Delta_{>0}
\]
and then $\sigma=\sigma_{j'}$ for some $j'$, $0\leq j'\leq r-1$. This would imply that
\[
\sigma_j<\sigma_j+\sigma_{j'}\leq\gamma<\Delta_{>0}.
\]
Since $\sigma_j+\sigma_{j'}\in \Sigma_{>0}$ and $\sigma_j+\sigma_{j'}\leq\gamma<\Delta_{>0}$, this implies that $\sigma_j+\sigma_{j'}=\sigma_{j''}$ for some $j''$, $j< j''\leq r-1$. This is a contradiction to the fact that $j$ was chosen as the largest index for which $\sigma_j\leq\gamma$. We have proved that for every $\gamma\in\Gamma_{\geq 0}$, if $\gamma<\Delta_{>0}$, then there exist $i$ and $j$, $0\leq i\leq s-1$ and $0\leq j\leq r-1$, such that $\gamma=\sigma_j+\gamma_i$. Therefore,
\[
\epsilon(\Gamma\mid\Delta)\leq r\cdot s=\epsilon(\Gamma\mid\Sigma)\cdot\epsilon(\Sigma\mid\Delta).
\]
This concludes the proof that $\epsilon(\Gamma\mid\Delta)=\epsilon(\Gamma\mid\Sigma)\cdot\epsilon(\Sigma\mid\Delta)$.

Since $\epsilon(\Gamma\mid\Sigma)\leq (\Gamma:\Sigma)$, $\epsilon(\Sigma:\Delta)\leq (\Sigma:\Delta)$ and $\epsilon(\Gamma:\Delta)\leq (\Gamma:\Delta)$, we have
\[
\epsilon(\Gamma:\Delta)=\epsilon(\Gamma\mid\Sigma)\cdot\epsilon(\Sigma\mid\Delta)\leq (\Gamma:\Sigma)\cdot(\Sigma:\Delta)=(\Gamma:\Delta)\leq\epsilon(\Gamma:\Delta).
\]
If the equality holds in one of the above inequalities, then equality holds everywhere. Therefore, $\epsilon(\Gamma\mid\Delta)=(\Gamma:\Delta)$ if and only if $\epsilon(\Gamma\mid\Sigma)=(\Gamma:\Sigma)$ and $\epsilon(\Sigma\mid\Delta)=(\Sigma:\Delta)$.
\end{proof}

\begin{Teo}[\ding{183} $\Llr$ \ding{184}] We have that $\textrm{gr}_\omega(\VR_\omega)$ is finitely generated over $\textrm{gr}_\nu(\VR_\nu)$ if and only if $\left(\Gamma_\omega\right)_{\geq 0}$ is finitely generated over $\left(\Gamma_\nu\right)_{\geq 0}$.
\end{Teo}
\begin{proof}
Assume first that $\textrm{gr}_\omega(\VR_\omega)$ is finitely generated over $\textrm{gr}_\nu(\VR_\nu)$. This implies that there exist $f_1,\ldots,f_n\in \VR_\omega$ such that $\textrm{in}_\omega(f_1),\ldots,\textrm{in}_\omega(f_n)$ generate $\textrm{gr}_\omega(\VR_\omega)$ over $\textrm{gr}_\nu(\VR_\nu)$. For $\gamma\in (\Gamma_\omega)_{\geq 0}$ there exists $x\in\VR_\omega$ such that $\omega(x)=\gamma$. Since $\textrm{in}_\omega(f_1),\ldots,\textrm{in}_\omega(f_n)$ generate $\textrm{gr}_\omega(\VR_\omega)$, there exist $g_{11},\ldots,g_{1r_1},\ldots,g_{1n},\ldots,g_{nr_n}\in\VR_\nu$, such that
\begin{displaymath}
\begin{array}{ccl}
\textrm{in}_\omega(x)&=&\left(\displaystyle\sum_{j=1}^{r_1}\textrm{in}_\nu(g_{1j})\right)\textrm{in}_\omega(f_1)+\ldots+\left(\displaystyle\sum_{j=1}^{r_n}\textrm{in}_\nu(g_{nj})\right)\textrm{in}_\omega(f_n)\\
&=&\left(\displaystyle\sum_{j=1}^{r_1}\textrm{in}_\nu(g_{1j})\textrm{in}_\omega(f_1)\right)+\ldots+\left(\displaystyle\sum_{j=1}^{r_n}\textrm{in}_\nu(g_{nj})\textrm{in}_\omega(f_n)\right)
\end{array}
\end{displaymath}
By the definition of the graded algebra, this means that we just need to consider the homogeneous components of degree $\gamma$ in the equation above, hence there exists some $f_i$ and $g_{ij}$ such that
\[
\gamma=\omega(x)=\omega(g_{ij}f_i)=\nu(g_{ij})+\omega(f_i).
\]
Hence
\[
\left(\Gamma_\omega\right)_{\geq 0}=\bigcup_{i=1}^n\left(\omega(f_i)+\left(\Gamma_\nu\right)_{\geq 0}\right).
\]

For the converse, assume that $\left(\Gamma_\omega\right)_{\geq 0}$ has finite index over $\left(\Gamma_\nu\right)_{\geq 0}$. Thus there exist $f_1,\ldots,f_n\in \VR_\omega$ such that
\[
\left(\Gamma_\omega\right)_{\geq 0}=\bigcup_{i=1}^n\left(\omega(f_i)+\left(\Gamma_\nu\right)_{\geq 0}\right).
\]
Moreover, since $F\omega$ is a finite $K\nu$-vector space, there exists $h_1,\ldots,h_r\in \VR_\omega$ such that $h_1\omega,\ldots,h_r\omega$ is a $K\nu$-basis of $L\omega$. For $x\in\VR_\omega$ there exists $i_0\in\{1,\ldots,n\}$ such that $\omega(x)=\nu(g_{i_0})+\omega(f_{i_0})$ for some $g_{i_0}\in\VR_\nu$. This means that $\omega\left(\frac{x}{g_{i_0}f_{i_0}}\right)=0$. Then, there exists $a_1,\ldots,a_r\in \VR_\nu$ such that
\[
\left(\frac{x}{g_{i_0}f_{i_0}}\right)\omega=a_1\nu\cdot h_1\omega+\ldots+a_r\nu\cdot h_r\omega.
\]
Hence
\[
\omega(x-a_1 h_1g_{i_0}f_{i_0}-\cdots-a_r h_rg_{i_0}f_{i_0})>\omega(x).
\]
Consequently,
\[
\textrm{in}_\omega(x)=\sum_{j\in J}\textrm{in}_\nu(a_jg_{i_0})\textrm{in}_\omega(h_jf_{i_0}),
\]
where
\[
J=\{j\in \{1,\ldots,r\}\mid \omega (a_1 h_1g_{i_0}f_{i_0})=\omega(x)\}.
\]
Therefore, $\textrm{gr}_\omega(\VR_\omega)$ is generated over $\textrm{gr}_\nu(\VR_\nu)$ by
\[
\{\textrm{in}_\omega(h_jf_{i})\mid 1\leq j\leq r\mbox{ and }1\leq i\leq n\}.
\]
\end{proof}

\section{Proof of the necessity condition}\label{Necessity}

The statement and proof of Theorem \ref{Knaf} are by Hagen Knaf.

\begin{Teo}(\ding{182} $\Lra$ \ding{188})\label{Knaf}(Knaf)
Assume that we are in condition (\ref{sit}). Then for the local ring extension $\VR_\nu\subseteq\VR_\omega$ to be essentially finitely generated the following conditions are necessary:
\[
e(\omega|\nu) = \epsilon(\omega|\nu)\mbox{ and }d(\omega|\nu) = 1.
\]
\end{Teo}

The proof given here is based on three results, which are interesting for their own sake. Comparing the construction of the henselization $K^h$ of a valued field $(K,\nu)$ in \cite[Section 17]{End} and the construction of the henselization of a normal local ring in \cite[Section 43]{Nag}, we see that the valuation ring of the valued field $K^h$ is the henselization $\VR_{\nu}^h$ of $\VR_\nu$.

\begin{Lema}\label{Lemafinringepsl}
Assume that we are in condition (\ref{sit}). Then the ring $\VR_\omega/\MI_\nu\VR_\omega$ is a $K\nu$-vector space of dimension $\epsilon(\omega|\nu)f(\omega|\nu)$.
\end{Lema}

The proof can be found in \cite{End} (18.5).

\begin{Lema}\label{Lemma23}
Assume that we are in condition (\ref{sit}). If $\VR_\nu\subseteq \VR_\omega$ is a finite ring extension, then $e(\omega|\nu) = \epsilon(\omega|\nu)$ and $d(\omega|\nu) = 1$ hold.
\end{Lema}
\begin{proof}
The finitely generated, torsion-free $\VR_\nu$-module $\VR_\omega$ is free of rank $[L : K]$ by \cite[(18.6)]{End}. By Nakayama's lemma, lemma \ref{Lemafinringepsl} and \cite[(18.3)]{End} one gets
\[
[L : K] = \epsilon(\omega|\nu)f (\omega|\nu) \leq d(\omega|\nu)e(\omega|\nu)f (\omega|\nu) = [L^h : K^h ] \leq [L : K]
\]
and thus the assertion.
\end{proof}

\begin{Lema}\label{Lemmaalgextcomp}
Assume that we are in condition (\ref{sit}). Then $\VR_\omega^h=\VR_\omega\cdot \VR_\nu^h$, where the ring compositum is formed within the field $L^h$.
\end{Lema}
\begin{proof}
Recall that a local ring $R$ is called henselian, if Hensel's lemma holds for that ring. The henselization $R^h$ of a local ring is henselian \cite{Nag}, (43.3), or just for valuation domains \cite{End}, Theorem 17.11. The latter suffices for the current proof. The extension $\VR_\nu^h\subseteq \VR_\omega\cdot \VR_\nu^h$ is integral and local, since $\VR_\omega^h$ is the integral closure of $\VR_\nu^h$ in $L^h$. Therefore by \cite{Nag}, Corollary 43.13 the ring $\VR_\omega\cdot \VR_\nu^h$ is henselian, which by \cite{Nag}, Theorem 43.5 yields the existence of an injective $\VR_\omega$-algebra homomorphism $f : \VR_\omega^h \lra \VR_\omega\cdot\VR_\nu^h$. In particular, $\VR_\omega \cdot \VR_\nu^h$ is an overring of the valuation ring $f (O_\omega^h)$ within $L^h$ and is thus a valuation ring itself. 
Since $\VR_\omega\cdot \VR_\nu^h\subseteq \VR_\omega^h$ is integral the assertion follows.
\end{proof}

\begin{proof}[Proof of Theorem \ref{Knaf}]
By assumption there exists a finitely generated $\VR_\nu$-algebra $A\subseteq \VR_\omega$ and a multiplicative subset $S \subseteq A$ such that $\VR_\omega=S^{-1}A$. The prime $\mathfrak{q}=\MI_\omega\cap A$ has the property $\mathfrak{q}\cap S=\emptyset$, hence $A_\mathfrak{q} = \VR_\omega$. The ring compositum $A\cdot \VR_\nu^h$ formed within $L^h$ is local, because it lies between $\VR_\nu^h$ and $\VR_\omega^h$. Its maximal ideal $\MI$ satisfies $\MI\cap A= \mathfrak{q}$, because on the one hand $\MI=\MI_\omega\cap (A \cdot\VR_\nu^h)$ and on the other hand $\MI_\omega^h\cap A=\mathfrak q$. Consequently $A_\mathfrak q\subseteq A\cdot\VR_\nu^h$, which by assumption and lemma \ref{Lemmaalgextcomp} implies
$A\cdot \VR_\nu^h=\VR_\omega^h$. In particular, by Proposition \ref{modvsalg}, $\VR_\nu^h\subseteq \VR_\omega^h$ is a finite ring extension, which by Lemma \ref{Lemma23} implies $e(\omega^h |\nu^h ) = \epsilon(\omega^h |\nu^h)$ and $d(\omega^h|\nu^h) = 1$. Since the henselization is an immediate extension, we have that $e(\omega^h|\nu^h)=e(\omega|\nu)$ and $\epsilon(\omega^h|\nu^h) =\epsilon(\omega|\nu)$. Moreover by definition $d(\omega^h|\nu^h) = d(\omega|\nu)$.
\end{proof}

\section{Reduction of Conjecture \ref{MainConj} to the separable case}\label{Secredsep}
\begin{Lema}\label{Lemamtransessentfinie}
Let $K\subseteq F\subseteq L$ be finite extensions of fields and let $\omega$ be a valuation on $L$ (whose restrictions to $F$ and $K$ we will denote by $\mu$ and $\nu$, respectively). If $\VR_\omega$ is essentially finitely generated over $\VR_{\mu}$ and $\VR_{\mu}$ is essentially finitely generated over $\VR_\nu$, then $\VR_\omega$ is essentially finitely generated over $\VR_\nu$.
\end{Lema}
\begin{proof}
First we observe that if for some $R\subseteq \VR_\omega$ there exists a prime ideal $\mathfrak q$ of $R$ such that $\VR_\omega=R_\mathfrak q$, then $\VR_\omega=R_\omega$. Indeed, since $R_\mathfrak q=\VR_\omega$ we have that $R\setminus\mathfrak q\subseteq R\setminus \MI_\omega$. Therefore $\VR_\omega=R_\omega$.

Assume now that $\VR_\omega$ is essentially finitely generated over $\VR_{\mu}$ and that $\VR_{\mu}$ is essentially finitely generated over $\VR_\nu$. Then there exist $x_1,\ldots,x_r\in \VR_\omega$ and $y_1,\ldots,y_s\in \VR_{\mu}$ such that
\[
\VR_\omega=\VR_{\mu}[x_1,\ldots,x_r]_{\omega}\mbox{ and }\VR_{\mu}=\VR_\nu[y_1,\ldots,y_s]_{\mu}.
\]
Since $\VR_\nu[y_1,\ldots,y_s]\cap\MI_{\mu}=\VR_\nu[y_1,\ldots,y_s]\cap\MI_{\omega}$ this implies that
\[
\VR_\omega=\VR_{\nu}[x_1,\ldots,x_r,y_1,\ldots,y_s]_{\omega}.
\]
Therefore, $\VR_\omega$ is essentially finitely generated over $\VR_\nu$.
\end{proof}
\begin{Prop}\label{Propusred}
Assume that we are in situation (1). Let $L'$ be a subfield of $L$ containing $K$ and set $\omega'=\omega|_{L'}$. Assume that $\omega$ is the unique extension of $\omega'$ to $L$. If $\VR_{\omega'}$ is essentially finitely generated over $\VR_\nu$, $d(\omega|\omega')=1$ and $\epsilon(\omega|\omega')=e(\omega|\omega')$, then $\VR_\omega$ is essentially finitely generated over $\VR_\nu$.
\end{Prop}

\begin{proof}
Since $\omega$ is the only extension of $\omega'$ to $L$, $d(\omega|\omega')=1$ and $\epsilon(\omega|\omega') = e(\omega|\omega')$, Corollary \ref{Endler1} guarantees that the integral closure $D'$ of $\VR_{\omega'}$ in $L$ is a finite $\VR_{\omega'}$-module. Since $O_\omega = D'_{\omega}$, this implies that $\VR_\omega$ is finitely generated over $\VR_{\omega'}$. By our assumption and Lemma \ref{Lemamtransessentfinie} we conclude that $\VR_\omega$ is essentially finitely generated over $\VR_\nu$.
\end{proof}

\begin{Cor}
If Conjecture \ref{Conjectsep} is true, then Conjecture \ref{MainConj} is also true.
\end{Cor}
\begin{proof}
Assume now that we are in situation (2) and set $L'$ to be the separable closure of $K$ in $L$ and $\omega'=\omega|_{L'}$. Since $L|L'$ is purely inseparable, $\omega$ is the unique extension of $\omega'$ to $L$.

Since $(L|K,\nu)$ is defectless, so are $(L|L',\omega')$ and $(L'|K,\nu)$. Since $\epsilon$ is multiplicative (Lemma \ref{epsilonmiult}) and $\epsilon(\omega|\nu) = e(\omega|\nu)$ then also $\epsilon(\omega|\omega')=e(\omega|\omega')$ and $\epsilon(\omega'|\nu)=e(\omega'|\nu)$. If Conjecture \ref{Conjectsep} is true, then we are in the situation of Proposition \ref{Propusred}, and therefore $\VR_\omega$ is essentially finitely generated over $\VR_\nu$.
\end{proof}

\section{Proof of Theorem \ref{Teocutko}}\label{Sectproofmanotehm}

We first review some notation that will be used in the proof. If $R$ is a local ring, we will denote its maximal ideal by $\mathfrak m_R$. The maximal ideal of a valuation ring $\mathcal O_{\nu}$ will be denoted by $\mathfrak m_{\nu}$. If a local domain $R$ is a subring of a local domain $S$ and $\MI_S\cap R=\mathfrak m_R$ we will say that $S$ dominates $R$ and that  $R\rightarrow S$ is dominating.  We will say that $\nu$ dominates $R$ if $\mathcal O_{\nu}$ dominates $R$.
Recall that if $R$ is a subring of a valuation ring $\mathcal O_{\nu}$, then $R_{\nu}$ is the localization $R_{\nu}=R_{\mathfrak m_{\nu}\cap R}$.

If a local domain $S$ dominates  a local domain $R$, $S$ is essentially of finite type over $R$ and $R$ and $S$ have the same quotient field, then $R\rightarrow S$ will be called a birational extension. If $R$ is a regular local ring, then a quadratic transform of $R$ is a dominating map $R\rightarrow S$ such that $S$ is a local ring of the blowup of the maximal ideal of $R$. A quadratic transform $R\rightarrow S$ is a quadratic transform along a valuation $\nu$ if $\nu$ dominates $S$. An iterated quadratic transform is a finite sequence of quadratic transforms.

We now present the proof of Theorem \ref{Teocutko}.

After possibly replacing $R$ with a birational extension of $R$ along $\nu$, we may assume that $R$ is normal and $\VR_{\nu}/\mathfrak m_{\nu}$ is algebraic over $R/\mathfrak m_R$. If $\dim R=1$, then $R$ is a valuation ring, so $R=\VR_{\nu}$. Since $R$ is excellent, the integral closure of $\VR_{\nu}$ in $L$ is a finite $R$-module. Thus $\VR_{\omega}$ is the localization of a finite $\VR_{\nu}$-module.

From now on, assume that $\dim R=2$.

 Let $S$ be the local ring  of the integral closure of $R$ in $L$ which is dominated by $\omega$. The ring $S$ is essentially of finite type over $R$ since $R$ is excellent.
 
First assume that the rank ${\rm rk}(\nu)>1$. By Abhyankar's inequality \cite[Theorem 1]{Ab}, this implies that ${\rm rk}(\nu)=2$ and 
\[
\Gamma_\omega\simeq (\Z^2)_{\rm lex}.
\]
By Theorem 3.7 of \cite{Ramif},  our assumption that ${\rm rk}(\nu)=2$ implies the assumption of our theorem that  $d(\omega|\nu)=1$, and that there exists a diagram
\begin{displaymath}
\begin{array}{ccc}
R_0& \longrightarrow & S_0\\
\big\uparrow&&\big\uparrow\\
R&\longrightarrow&S\\
\end{array}
\end{displaymath}
such that $R_0$ and $S_0$ are two-dimensional regular local rings, $\omega$ dominates $S_0$, all the arrows are dominating and the vertical arrows are birational extensions (i.e., localizations of blow-ups at ideals) and regular parameter $x_1(0), x_2(0)$ in $R_0$ and $y_1(0), y_2(0)$ in $S_0$ such that there exist units $\alpha, \beta\in S_0$ and $a,b,c,d\in\N$ such that $|ad-bc|=e$ and
\[
x_1(0)=\alpha\cdot y_1(0)^ay_2(0)^b\mbox{ and }x_2(0)=\beta\cdot y_1(0)^cy_2(0)^d.
\]
Further, the forms of these equations are stable under further quadratic transforms of $\omega$ along $\nu$. We have that 
$S_0$ is a local ring of a finitely generated $R_0$-algebra since $S$ is a localization of a finitely generated $R$-algebra and $S\lra S_0$ is a birational extension. That is, there exists $z_1,\ldots,z_r\in S_0$ such that $S_0=T_{\omega}$ where $T=R_0[z_1,\ldots,z_r]$.

Since $\Gamma_{\nu}$  has finite index in $\Gamma_{\omega}$, there exist $u,v\in K$ such that $\nu(u)>m\nu(v)>0$ for all positive integers $m$. Replacing $R_0$ and $S_0$ with suitable iterated quadratic transforms along $\nu$ and $\omega$ respectively, we may assume that $u,v\in R_0$ and 
$$
u=\phi x_1(0)^{s_1}x_2(0)^{s_2},\,\, v=\psi x_1(0)^{t_1}x_2(0)^{t_2}
$$
for some $s_1,s_2,t_1,t_2\in \N$ and units $\phi,\psi\in R_0$. After possibly interchanging $x_1(0)$ and $x_2(0)$, we then have that $\nu(x_1(0))>m\nu(x_2(0))>0$ for all positive integers $m$. Similarly, we can assume that $\omega(y_1(0))>m\omega(y_2(0))>0$ for all positive integers $m$. 

Since $\epsilon=e$, there exists $a\in \Z$ such that $(1,a)\in \Gamma_\nu$ (by Corollary \ref{sgroupcor}). Thus, after possibly performing some further iterated quadratic transforms along $\nu$ and $\omega$,  we can assume that there exists $f\in R_0$ such that $\omega(f)=(1,a)$ for some $a\in \Z$ and that $f=\gamma\cdot x_1(0)^mx_2(0)^n$ for some $m,n\in\N$, where $\gamma\in R_0$ is a unit. Then $m=1$ and $\nu(x_1(0))=(1,l)$ for some $l\in \Z$. Consequently, $\omega (y_1(0))=(1,m)$ for some $m\in \Z$ and $a=1$. Further, $c=0$. So
\[
x_1(0)=\alpha\cdot y_1(0)y_2(0)^b\mbox{ and }x_2(0)=\beta\cdot y_2(0)^e.
\]

Let $S_0\lra S_1$ be the iterated quadratic transform of $S_0$ along $\omega$ with
\[
S_1=S_0[y_1(1)]_\omega
\]
where
\[
y_1(0)=y_1(1)y_2(0)^r\mbox{ and }y_2(0)=y_2(1)
\]
and $r\in\Z_+$ is chosen so that $b+r=se$ for some $s\in \Z_+$. Let
\[
R_0\lra R_1=R_0[x_1(1)]_\nu
\]
be the iterated quadratic transform along $\nu$ defined by
\[
x_1(0)=x_1(1)x_2(0)^s\mbox{ and }x_2(0)=x_2(1).
\]
We have that
\[
x_1(1)=\frac{x_1(0)}{x_2(0)^s}=\alpha\cdot \beta^{-s}y_1(1)\mbox{ and }x_2(1)=\beta\cdot y_2(1)^e.
\]
In particular, $S_1$ dominates $R_1$. Further, $S_1$ is a localization of $R_1[z_1,\ldots,z_r]$.

Iterating this construction, we produce an infinite commutative diagram of regular local rings  dominated by $\omega$
\begin{displaymath}
\begin{array}{ccc}
\vdots&&\vdots\\
R_n& \longrightarrow & S_n\\
\uparrow&&\uparrow\\
\vdots&&\vdots\\
\uparrow&&\uparrow\\
R_1& \longrightarrow & S_1\\
\uparrow&&\uparrow\\
R_0&\longrightarrow&S_0\\
\end{array}
\end{displaymath}
such that $S_i$ is a localization of $R_i[z_1,\ldots,z_r]$ for every $i\in\N$ and the vertical arrows are iterated quadratic transforms. Since $R$ and $S$ are two-dimensional local domains, 
\[
\VR_\nu=\bigcup_{i=0}^\infty R_i\mbox{ and }\VR_\omega=\bigcup_{i=0}^\infty S_i
\]
by \cite[Lemma 12]{Ab}.
Therefore, $\VR_\omega$ is a localization of $\VR_\nu[z_1,\ldots,z_r]$.

Now we assume that ${\rm rk}(\nu)=1$. 
Since $d(\omega|\nu)=1$, by Theorem 3.3 and Theorem 3.7 of \cite{Ramif}, there exists a commutative diagram
\begin{displaymath}
\begin{array}{ccc}
R_0& \longrightarrow & S_0\\
\big\uparrow&&\big\uparrow\\
R&\longrightarrow&S\\
\end{array}
\end{displaymath}
such that $R_0$ and $S_0$ are two-dimensional regular local rings, $\omega$ dominates $S_0$, all the arrows are dominating and the vertical arrows are birational extensions and there exist regular parameters $x_1(0), x_2(0)$ in $R_0$ and $y_1(0), y_2(0)$ in $S_0$  such that, if the rational rank ${\rm ratrk}(\nu)=1$, then  there exists a unit $\alpha\in S_0$
such that
\begin{equation}\label{eqD3}
x_1(0)=\alpha\cdot y_1(0)^e\mbox{ and }x_2(0)= y_2(0),
\end{equation}
and if ${\rm ratrk}(\nu)=2$, then there exist units $\gamma$ and $\tau$ in $S_0$ such that 
\begin{equation}\label{eqD4}
x_1(0)=\gamma y_1(0)^ay_2(0)^b,\,\,x_2(0)=\tau y_1(0)^cy_2(0)^d
\end{equation}
where $|ad-bc|=e$. In both cases, we have that
  $\VR_\omega/\MI_\omega$ is the join of $\left(\VR_\nu/\MI_\nu\right)$ and $S_0/\MI_{S_0}$ and
\begin{equation}\label{eqD1}
f(\omega|\nu)=\left[S_0/\MI_{S_0}:R_0/\MI_{R_0}\right]=[\VR_\omega/\MI_\omega:\VR_\nu/\MI_\nu].
\end{equation}

Further, these equations are stable under suitable iterated quadratic transforms along $\nu$ and $\omega$.

Also, $S_0$ is a localization of a finitely generated $R$-algebra, since $S\lra S_0$ is a birational extension. Thus there exist $z_1,\ldots,z_r\in S_0$ such that $S_0=T_\omega$ where $T=R_0[z_1,\ldots,z_r]$.

We will treat the cases $e=1$ and $e>1$ separately. Assume first that $e=1$.  If ${\rm ratrk}(\nu)=1$, then replacing $y_1(0)$ with $\alpha\cdot y_1(0)$, we have that
\begin{equation}\label{eqpare1}
x_1(0)=y_1(0)\mbox{ and }x_2(0)=y_2(0).
\end{equation}
If ${\rm ratrk}(\nu)=2$, then from $|ad-bc|=1$, we have that there exists an iterated 	quadratic transform  $R_0\rightarrow R'$ along $\nu$ such that $S_0$ dominates $R'$ and $R'$ has regular parameters $x_1',x_2'$ such that $x_1'=\gamma' y_1(0)$ and $x_2'=\tau' y_2(0)$ where 
$\gamma',\tau'\in S_0$ are units. After replacing $R_0$ with $R'$, we then have that (\ref{eqpare1}) holds.

After possibly interchanging $x_1(0)$ and $x_2(0)$ (and $y_1(0)$ and $y_2(0)$) we can assume that $\nu(x_2(0))\geq \nu(x_1(0))$. Then the quadratic transform of $R_0$ along $\nu$ is
\[
R_0\lra R_1=R_0\left[\frac{x_2(0)}{x_1(0)}\right]_\nu.
\]
Let $x_1(1)=x_1(0)$. Then
\[
R_1/(x_1(1))\simeq R_0/\MI_{R_0}\left[\frac{x_2(0)}{x_1(0)}\right]_\mathfrak{n},
\]
where $\mathfrak{n}$ is a suitable maximal ideal and $\frac{x_2(0)}{x_1(0)}$ is transcendental over $R_0/\MI_{R_0}$. Thus, the maximal ideal of $R_1/(x_1(0))$ is generated by an element
\[
\overline{f}=\left(\frac{x_2(0)}{x_1(0)}\right)^d+\overline{a_1}\left(\frac{x_2(0)}{x_1(0)}\right)^{d-1}+\ldots+\overline{a_d}
\]
for some $d$ and $\overline a_i\in R_0/\mathfrak m_{R_0}$.
Then $R_1/\MI_{R_1}=R_0/\MI_{R_0}[z]$ where $z$ is the class of $\frac{x_2(0)}{x_1(0)}$ in $R_1/\MI_{R_1}\subseteq \VR_\nu/\MI_\nu\subseteq \VR_\omega/\MI_\omega$ and $\overline{f}$ is the minimal polynomial of $z$ over $R_0/\MI_{R_0}$. We have that $d=\left[R_1/\MI_{R_1}:R_0/\MI_{R_0}\right]$. The quadratic transform of $S_0$ along $\omega$ is
\[
S_0\lra S_1=S_0\left[\frac{y_2(0)}{y_1(0)}\right]_\omega,\mbox{ where }\frac{y_2(0)}{y_1(0)}=\frac{x_2(0)}{x_1(0)}.
\]
Let $y_1(1)=y_1(0)$. Now, as in the inclusion of $R_0\lra R_1$ we have
\[
S_1/\MI_{S_1}\simeq S_0/\MI_{S_0}[z]
\]
and the minimal polynomial of $z$ in $S_0/\MI_{S_0}\left[\frac{x_2(0)}{x_1(0)}\right]$ divides the minimal polynomial $\overline{f}$ of $z$ in $R_0/\MI_{S_0}\left[\frac{x_2(0)}{x_1(0)}\right]$. But by (\ref{eqD1}), we have
\[
\left[S_1/\MI_{S_1}:R_1/\MI_{R_1}\right]=\left[S_0/\MI_{S_0}:R_0/\MI_{R_0}\right]=f(\omega|\nu),
\]
so from
\[
\begin{array}{lll}
[S_1/\mathfrak m_{S_1}:R_0/\mathfrak m_{R_0}]
&=&[S_1/\mathfrak m_{S_1}:S_0/\mathfrak m_{S_0}]
[S_0/\mathfrak m_{S_0}:R_0/\mathfrak m_{R_0}]\\
&=&[S_1/\mathfrak m_{S_1}:R_1/\mathfrak m_{R_1}]
[R_1/\mathfrak m_{R_1}:R_0/\mathfrak m_{R_0}]
\end{array}
\]
we have that $[S_1/\mathfrak m_{S_1}:S_0/\mathfrak m_{S_0}]=d$ and so $\overline f$ is the minimal polynomial of $z$ over $S_0/\MI_{S_0}$. Now letting $x_2(1)=f$ where $f$ is a lifting of $\overline f$ to $R_1$, we have that $x_1(1), x_2(1)$ are regular parameters in $R_1$ and in $S_1$, giving an expresion like in (\ref{eqpare1}) in $R_1\lra S_1$.

Further, $S_1$ is a localization of $R_1[z_1,\ldots,z_r]$.  Iterating this construction, we have as in the case ${\rm rk}(\nu)=2$ that $\mathcal O_{\nu}=\cup_{i=0}^{\infty}R_i$ and $\mathcal O_{\omega}=\cup_{i=0}^{\infty} S_i$ are unions of iterated quadratic transforms of $R_0$ and $S_0$ along $\nu$ and $\omega$ respectively, and $S_i$ is a localization of $R_i[z_1,\ldots,z_r]$ for all $i$. Thus
 $\VR_\omega$ is a localization of $\VR_\nu[z_1,\ldots,z_r]$.

It remains to prove our theorem in the case when ${\rm rk}(\nu)=1$ and $\epsilon =e>1$. We then have that
$\nu$ is discrete of rank one by \cite[(18.4) b)]{End}. In this case we have
\[
\Gamma_\nu=e\Z\subseteq \Z=\Gamma_\omega.
\]

Then, there exists $g\in K$ such that $\nu(g)=e$ and we may assume (after possibly performing iterated quadratic transforms of $R_0$ and $S_0$ along $\nu$ and $\omega$) that $g\in R_0$ amd $g=\gamma\cdot x_1(0)^a$ for some unit $\gamma\in R_0$ and some $a\in \Z_{>0}$. Thus $a=1$ and $\nu(x_1(0))=e$, and so $\omega (y_1(0))=1$. Now $e\mid \omega(x_2(0))=\omega(y_2(0))$, so $\omega(x_2(0))=\omega(y_2(0))=es$ for some $s\in \Z_{>0}$. Thus
\[
\omega\left(\frac{y_2(0)}{y_1(0)}\right)=\omega\left(\frac{x_2(0)}{x_1(0)}\right)\geq 0.
\]
Let $R_1=R_0\left[\frac{x_2(0)}{x_1(0)}\right]_\nu$ and $S_1=S_0\left[\frac{y_2(0)}{y_1(0)}\right]_\omega$. Now $R_0\lra R_1$ is a quadratic transform along $\nu$ and $S_0\lra S_1$ is a  quadratic transform along $\omega$. As in the previous case, $S_1$ dominates $R_1$ and there exists $x_2(1)\in R_1$ such that $x_1(1)=x_1(0)$ and $x_2(1)$ are a regular system of parameters in $R_1$ such that $y_1(1)=y_1(0)$ and $y_2(1)=x_2(1)$ are regular system of parameters in $S_1$ satisfying $x_1(1)=\alpha\cdot y_1(1)^e$ and $x_2(1)=y_2(1)$ by (\ref{eqD3}). In particular, $S_1$ is a localization of $R_1[z_1,\ldots,z_r]$.

 Iterating this construction, as in the cases of ${\rm rk}(\nu)=2$ and ${\rm rk}(\nu)=1$ with $e=1$, we have that
$\mathcal O_{\nu}=\cup_{i=0}^{\infty}R_i$ and $\mathcal O_{\omega}=\cup_{i=0}^{\infty}S_i$ are unions of iterated quadratic transforms of $R_0$ and $S_0$ along $\nu$ and $\omega$ respectively., and $S_i$ is a localization of $R_i[z_1,\ldots,z_r]$ for all $i$. Thus $\mathcal O_{\omega}$ is a localization of $\mathcal O_{\nu}[z_1,\ldots,z_r]$.

\section{Proof of Theorem \ref{TeoAbh}}\label{SecAbh}

\begin{Prop}\label{PropGS} Suppose that $(K,\nu)$ is a valued field, $L$ is a finite extension field of $K$ and  $\omega$ is an extension of $\nu$ to $L$ such that 
$$
1<\epsilon(\omega|\nu)=e(\omega|\nu).
$$
Let $\Gamma_{\nu,1}$ be the first convex subgroup of $\Gamma_{\nu}$ and $\Gamma_{\omega,1}$ be the first convex subgroup of $\Gamma_{\omega}$. Then $\Gamma_{\omega,1}\cong \Z$ and in the short exact sequence of groups
\begin{equation}\label{eqAb3}
0\rightarrow \Gamma_{\omega,1}/\Gamma_{\nu,1}\rightarrow \Gamma_{\omega}/\Gamma_{\nu}\rightarrow (\Gamma_{\omega}/\Gamma_{\omega,1})/(\Gamma_{\nu}/\Gamma_{\nu,1})\rightarrow 0
\end{equation}
we have that 
$$
(\Gamma_{\omega}/\Gamma_{\omega,1})/(\Gamma_{\nu}/\Gamma_{\nu,1})=0
$$
and 
$$
\Gamma_{\omega}/\Gamma_{\nu}\cong \Gamma_{\omega,1}/\Gamma_{\nu,1}\cong \Z_e.
$$
\end{Prop}

\begin{proof} We have that $\Gamma_{\omega,1}\cong \Z$ by Lemma \ref{Lemma1CS}. The proposition now follows from Proposition \ref{lemmagammakk1}, Remark \ref{RemCS} and (\ref{eqAb3}).
\end{proof}

Suppose that $K$ is an algebraic function field over a field $k$. An algebraic local ring of $K$ is a local domain $R$ such that $R$ is essentially of finite type over $k$ and the quotient field of $R$ is $K$.

\begin{Teo}(\cite[Theorem 1.1]{KK}) \label{TheoremAbU}
Let $K$ be an algebraic function field over a field $k$, and let $\nu$ be an Abhyankar valuation on $K$.  Suppose that $\VR_\nu/\mathfrak m_\nu$ is separable over $k$ and $Z\subset \VR_\nu$  is a finite set. Then there exists an algebraic regular local ring $R$ of $K$ such that $\nu$ dominates $R$ and $\dim R={\rm ratrk}(\nu)$. Further, there exists a regular system of parameters in $R$ such that each element of $Z$ is a monomial in the regular system of parameters times a unit in $R$.
\end{Teo}

\subsection{Abhyankar valuations on algebraic function fields}
In this subsection, suppose  that  $K$ is an algebraic function field over a field $k$ and $\nu$ is an Abhyankar valuation on $K$. Then $\Gamma_{\nu}$ is a finitely generated (torsion free) abelian group by \cite[Lemma 1]{Ab}. Let $n={\rm ratrk}(\nu)$, and
$$
0=\Gamma_{\nu,0}\subset \Gamma_{\nu,1}\subset \cdots \subset \Gamma_{\nu,r}=\Gamma_{\nu}
$$
be the chain of convex subgroups of $\Gamma_{\nu}$. Each quotient $\Gamma_{\nu_i}/\Gamma_{\nu,i+1}$ is a torsion free abelian group. 

\begin{Prop}\label{CorAbU} In the conclusions of Theorem \ref{TheoremAbU}, we can choose $R$ so that $R$ has a regular system of parameters $\{z_{i,j}\}$ such that for each fixed $i$ with $1\le i\le r$, we have that $\{\nu(z_{i,j})\}$ is a $\Z$-basis of $\Gamma_{\nu,i}/\Gamma_{\nu,i-1}$.
\end{Prop}

\begin{proof}
There exists a $\Z$-basis $\{\gamma_{i,j}\}$ of $\Gamma_{\nu}$ such that $\gamma_{i,j}\in (\Gamma_{\nu})_{\ge 0}$ for all $i,j$ and for each fixed $i$ with $1\le i\le r$, the images of $\gamma_{i,j}$ in $\Gamma_{\nu,i}/\Gamma_{\nu,i-1}$ form a $\Z$-basis of $\Gamma_{\nu,i}/\Gamma_{\nu,i-1}$.

There exist $f_{i,j}\in \VR_\nu$ such that $\nu(f_{i,j})=\gamma_{i,j}$ for all $i,j$. Reindex the $f_{i,j}$ as $f_1,\ldots,f_n$ so that the $\nu(f_i)$ are increasing. By Theorem \ref{TheoremAbU}, there exists an algebraic regular local ring $R$ of $K$ such that $\dim R=n$ and there exists a regular system of parameters $z_1,\ldots,z_n$ in $R$ and units $\lambda_i\in R$ such that
$$
f_i=\lambda_iz_1^{a_1,i}\cdots z_n^{a_{n,i}}\mbox{ for }1\le i\le n.
$$
Since $\nu(f_1),\ldots,\nu(f_n)$ is a $\Z$-basis of $\Gamma_{\nu}$, we have that $\nu(z_1),\ldots,\nu(z_n)$ is a $\Z$-basis of $\Gamma_{\nu}$. Further, we can  reindex the $z_i$ as $z_{i,j}$ so that for all fixed $i$ with $1\le i\le r$, $\{\nu(z_{i,j})\}$ is a $\Z$-basis of $\Gamma_{\nu,i}/\Gamma_{\nu,i-1}$.
\end{proof}

Suppose that $R$ satisfies the conclusions of Proposition \ref{CorAbU}. Index the regular system of parameters $z_{i,j}$ as $x_1,\ldots,x_n$ so that $\nu(x_i)<\nu(x_j)$ if $i<j$.

We define a primitive monoidal transform (PMT) $R\rightarrow R_1$ along $\nu$ by
$R_1=R\left[\frac{x_j}{x_i}\right]_\nu$. We have that $R_1$ is a regular local ring with regular parameters $x_1(1),\ldots,x_n(1)$ defined by
$$
x_k=x_k(1)\mbox{ if $k\ne j$ and $x_j=x_j(1)x_i(1)$.}
$$
Further, $\{\nu(x_k(1))\mid 1\le k\le n\}$ is a $\Z$-basis of $\Gamma_{\nu}$, which satisfies the conclusions of Proposition \ref{CorAbU} in $R_1$.

We will find the following proposition useful.

\begin{Prop}\label{PropZa} Suppose that $R$ satisfies the conclusions of Proposition
 \ref{CorAbU}, with regular parameters $x_1,\ldots,x_n$  and $M_1=x_1^{a_1}\cdots x_n^{a_n}$, $M_2=x_1^{b_1}\cdots x_n^{b_n}$ are monomials such that $\nu(M_1)<\nu(M_2)$. Then there exists a sequence of PMTs along $\nu$
 $$
R\rightarrow R_1\rightarrow \cdots\rightarrow R_s
$$
such that $M_1$ divides $M_s$ in $R_s$.
\end{Prop}

\begin{proof} Consider the indexing $z_{i,j}$ of the regular parameters $x_i$ of the statement of Proposition \ref{CorAbU}. Write
$$
M_1=\prod_{i,j}z_{i,j}^{a_{i,j}},\,\, M_2=\prod_{i,j}z_{i,j}^{b_{i,j}}.
$$
There exists a largest index $l$ such that $\prod_jz_{l,j}^{a_{l,j}}\ne \prod_jz_{l,j}^{b_{l,j}}$.
Then $\nu(\prod_jz_{l,j}^{a_{l.j}})<\nu(\prod_jz_{l,j}^{b_{l,j}})$.  By \cite[Theorem 2]{Z}, there exists a sequence of PMTs  $R\rightarrow R_s$ along $\nu$ in the  variables  $z_{l,j}(m)$  from the regular parameters  of $R_m$ as $j$ varies, 
such that $\prod_jz_{l,j}^{a_{l.j}}$ divides $\prod_jz_{l,j}^{b_{l,j}}$ in $R_s$. Writing $M_1$ and $M_2$ in the regular parameters $z_{i,j}(s)$ of $R_s$ as
$$
M_1=\prod z_{i,j}(s)^{a_{i,j}(s)}\mbox{ and }M_2=\prod z_{i,j}(s)^{b_{i,j}(s)},
$$
we have that 
$$
M_2=\left(\prod_{i<l,j}z_{i,j}(s)^{b_{i,j}(s)}\right)\left(\prod_jz_{l,j}(s)^{b_{l,j}(s)-a_{l,j}(s)}\right)
\left(\prod_{i>l,j}z_{l,j}(s)^{a_{i,j}(s)}\right)
$$
with $b_{l,j}(s)-a_{l,j}(s)\ge 0$ for all $j$ and for some $j$, $b_{l,j}(s)-a_{l,j}(s)> 0$.
Without loss of generality, this occurs for $j=1$.

Now perform a  sequence of PMTs $R_s\rightarrow R_m$ along $\nu$ defined by
$z_{l,1}(t+1)=z_{l,1}(t)z_{i,j}(t)$ for $i<l$ and $j$ such that $b_{i,j}(t)>a_{i,j}(t)$ where 
$$
M_1=\prod z_{i,j}(t)^{a_{i,j}(t)}\mbox{ and }M_2=\prod z_{i,j}(t)^{b_{i,j}(t)}
$$
to achieve that $M_1$ divides $M_2$ in $R_m$.
\end{proof}

\subsection{Abhyankar valuations  in finite extensions}
We continue the notation of the previous section, and further suppose that $L$ is a finite extension field of $K$ and $\omega$ is an extension of $\nu$ to $L$, such that $\VR_\omega/\mathfrak m_\omega$ is separable over $k$. We have that $\omega$ is also an Abhyankar valuation and $d(\omega|\nu)=1$ by \cite[Theorem 1]{KK2}.

We suppose that $\epsilon(\omega|\nu)=e(\omega|\nu)$.  

\begin{Prop}\label{PropGoodForm} There exist algebraic regular local rings $R$ of $K$ and $S$ of $L$ which are dominated by $\omega$ and $\nu$ respectively such that $S$ dominates $R$ and $R$ has regular parameters $x_1,\ldots,x_n$ and $S$ has regular parameters $y_1,\ldots,y_n$ satisfying the conclusions of Proposition \ref{CorAbU} such that there is an expression
$$
x_1=\gamma y_1^e\mbox{ and }
x_i=y_i\mbox{ for }2\le i\le n
$$
where $\gamma$ is a unit in $S$.
\end{Prop}

\begin{proof} By  Theorem \ref{TheoremAbU} and Proposition \ref{CorAbU} there exist algebraic regular local rings $R_0$ of $K$ and $S_0$ of $L$ such that 
 $\omega$ dominates $S_0$, $S_0$ dominates $R_0$, $R_0$ has  regular parameters $x_1,\ldots,x_n$ and $S_0$ has regular parameters $y_1,\ldots,y_n$ satisfying the conclusions of Proposition \ref{CorAbU} and there exist units $\gamma_i\in S$ such that 
 \begin{equation}\label{eqAb1} 
 x_i=\gamma_iy_1^{a_{i,1}}\cdots y_n^{a_{i,n}}
 \end{equation}
 for $1\le i\le n$. Since $\nu(x_1),\ldots,\nu(x_n)$ is a $\Z$-basis of $\Gamma_{\nu}$ and $\omega(y_1),\ldots,\omega(y_n)$ is a $\Z$-basis of $\Gamma_{\omega}$, we have that 
 $$
 e=|{\rm Det}(A)|
 $$
 where $A=(a_{ij})$ is the $n\times n$ matrix of exponents in (\ref{eqAb1}).
 
 First suppose that  $\epsilon(\omega|\nu)=e(\omega|\nu)>1$. Then $\Gamma_{\omega,1}\cong \Z$ and $\Gamma_{\omega,1}/\Gamma_{\nu,1}\cong \Z_e$ by Proposition \ref{PropGS}.
 Thus $\nu(x_1)$ is a $\Z$-basis of $\Gamma_{\nu,1}$ and $\nu(y_1)$ is a $\Z$-basis of $\Gamma_{\omega,1}$. We thus have that $a_{1,j}=0$ for $j>1$ and $a_{1,1}$ is a positive multiple of $e$. Thus from $|{\rm Det}(A)|=e$ we have that $a_{1,1}=e$ and $|{\rm Det}(\overline A)|=1$ where 
 $$
 \overline A=
 \left(\begin{array}{lll}
 a_{2,2}&\cdots&a_{2,n}\\
 &\vdots&\\
 a_{n,2}&\cdots&a_{n,n}
 \end{array}\right).
 $$
 Since ${\rm Det}(\overline A)=\pm 1$, there exist $r_2,\ldots,r_n\in \Z$ such that 
 $$
 \overline A\left(\begin{array}{c}r_2\\ \vdots\\ r_n\end{array}\right)=-\left(\begin{array}{c}
 a_{2,1}\\ \vdots \\a_{n,1}\end{array}\right).
 $$
 Define $s_2,\ldots,s_n\in \Z_{\ge 0}$ by
 $$
 \left(\begin{array}{c} s_2\\ \vdots\\ s_n\end{array}\right)=\overline A
 \left(\begin{array}{c} 1\\ \vdots\\ 1\end{array}\right).
 $$
 There exists $r\in \Z_{>0}$ such that $r_i+te>0$ for all $i$. Perform the sequence of PMTs $S_0\rightarrow S_1$ along $\omega$ defined by
 $$
 y_1= y_1(1), y_i=y_i(1)y_1(1)^{r_i+te}\mbox{ for }2\le i\le n.
 $$
 We have that $S_1=S_0[y_1(1),\ldots,y_n(1)]_{\omega}$ is a regular local ring with regular parameters $y_1(1),\ldots,y_n(1)$ which dominates $R_0$ and there exist units $\gamma_i'\in S_1$ such that 
 $$
 x_1=\gamma_1'y_1(1)^e\mbox{ and }x_i=\gamma_i'y_1(1)^{ets_i}y_2(1)^{a_{2,1}}\cdots y_n(1)^{a_{i,n}}\mbox{ for }2\le i\le n.
 $$
 Now perform the sequence of PMTs $R_0\rightarrow R_1$ along $\nu$ defined by
 $$
 x_1=x_1(1), x_i=x_1(1)^{ts_i}x_i(1)\mbox{ for }2\le i\le n.
 $$
 We have that $S_1$ dominates $R_1$ and there exist units $\gamma_i(1)\in S_1$ such that
 \begin{equation}\label{eqAb2}
 x_1(1)=\gamma_1(1)y_1(1)^e, x_i(1)=\gamma_i(1)y_2^{a_{i,2}}\cdots y_n(1)^{a_{i,n}}\mbox{ for }2\le i\le n.
 \end{equation}
 We continue to have ${\rm Det}(\overline A)=\pm 1$.  Let $B=\overline A^{-1}$, Write
 $$
 B=\left(\begin{array}{lll}
 b_{2,2}&\cdots&b_{2,n}\\
 &\vdots&\\
 b_{n,2}&\cdots&b_{n,n}
 \end{array}
 \right)
 $$
 with $b_{i,j}\in \Z$. We now replace the $y_i(1)$ with the product of the unit
 $\gamma_2(1)^{-b_{i,2}}\cdots \gamma_n(1)^{-b_{i,n}}$ and $y_i(1)$ for $2\le i\le n$ to get $\gamma_i(1)=1$ for $2\le i\le n$ in (\ref{eqAb2}).
 
Now define a birational transformation $R_1\rightarrow R_2$ along $\nu$ by $R_2=R_1[x_1(2),\ldots,x_n(2)]_{\nu}$ where
$$
x_n(1)=x_n(2)\mbox{ and }x_i(1)=x_2(2)^{a_{i,2}}\cdots x_n(2)^{a_{i,n}}\mbox{ for }2\le i\le n.
$$ 
The ring $R_2$ is a regular local ring with regular parameters $x_1(2),\ldots,x_n(2)$. We have that $R_2$ is dominated by $S_1$, and
$$
x_1(2)=\gamma y_1(2)^e\mbox{ and }x_i(2) = y_i(1)\mbox{ for }2\le i\le n
$$
where $\gamma$ is a unit in $S_1$, giving the conclusions of the proposition.

Now suppose that $e=1$. This case is much simpler.  In (\ref{eqAb1}) we then have that ${\rm Det}(A)=\pm 1$. Taking $B=A^{-1}=(b_{i,j})$, we can then make the change of variables in $S_0$ 
replacing the $y_i$ with the product of the unit
 $\gamma_1^{-b_{i,1}}\cdots \gamma_n^{-b_{i,n}}$ times $y_i$ for $1\le i\le n$ to get $\gamma_i=1$ for $1\le i\le n$ in (\ref{eqAb1}).
 
 Now define a birational transformation $R_0\rightarrow R_1$ along $\nu$ by $R_1=R_0[x_1(1),\ldots,x_n(1)]_{\nu}$ where
$$
x_i=x_1(1)^{a_{i,1}}\cdots x_n(1)^{a_{i,n}}\mbox{ for }1\le i\le n.
$$ 
The ring $R_1$ is a regular local ring with regular parameters $x_1(1),\ldots,x_n(1)$. We have that $R_1$ is dominated by $S$, and
$$
x_i(1) = y_i(1)\mbox{ for }1\le i\le n,
$$
giving the  conclusions of the proposition.
\end{proof}

\begin{Prop}\label{PropSR} Suppose that $R\rightarrow S$ has the form of the conclusions of Proposition \ref{PropGoodForm}   and $1\le i<j\le n$ (with $\nu(x_i)<\nu(x_j)$). Then there exist $z_1,\ldots,z_m\in \VR_\omega$ such that $S_0=R_0[z_1,\ldots,z_m]_\omega$.

Let $R_0\rightarrow R_1=R_0[x_1(1),\ldots,x_n(1)]_\nu$ be the PMT along $\nu$ defined by 
$$
x_k=x_k(1)\mbox{ if }i\ne j\mbox{ and }x_j=x_j(1)x_i(1).
$$
Then there exists a sequence of PMTs along $\omega$, $S_0\rightarrow S_1$ such that $S_1$ dominates $R_1$ and  $S_1$ has regular parameters $y_1(1),\ldots,y_n(1)$ such that
$$
x_1(1)=\gamma y_1(1)^e\mbox{ and }x_i(1)=y_i(1)\mbox{ for }2\le i\le n
$$
and $S_1=R_1[z_1,\ldots,z_m]_{\omega}$.
\end{Prop}

\begin{proof} The expression $S_0=R_0[z_1,\ldots,z_m]_\omega$ follows since $S_0$ is essentially of finite type over $k$.

If $i\ne 1$ we have that both $i=1$ and $e=1$, then $S_1$ is defined by  
$$
y_k=y_k(1)\mbox{ if }i\ne j\mbox{ and }y_j=y_j(1)y_i(1).
$$
If $i=1$ and $e>1$, then define $S\rightarrow S_1$  by the sequence of PMTs along $\omega$ 
$$
y_k=y_k(1)\mbox{ if }k\ne j\mbox{ and }y_j=y_1(1)^ey_j(1).
$$
\end{proof}

We now prove Theorem \ref{TeoAbh}. Let $R_0\rightarrow S_0$ have the form of the conclusions of Proposition \ref{PropGoodForm}, and write $S_0=R_0[z_1,\ldots,z_m]_{\omega}$.

Suppose that $f\in \VR_\omega$. Write $f=\frac{g}{h}$ with $g,h\in S_0$. Let $\kappa$ be a coefficient field of the $\mathfrak m_{S_0}$-adic completion $\widehat{S_0}$ of $S_0$. We then have that $\widehat S_0$ is the power series ring $\widehat{S_0}=\kappa[[y_1,\ldots,y_n]]$. Expand $g=\sum \alpha_{i_1,\ldots,i_n}y_1^{i_1}\cdots y_n^{i_n}$ and $h=\sum \beta_{j_1,\ldots,j_n}y_1^{j_1}\cdots y_n^{j_n}$ with $\alpha_{i_1,\ldots,i_n},\beta_{j_1,\ldots,j_n}\in \kappa$. We have that
$$
\omega(g)=\min\{\omega(y_1^{i_1}\cdots y_n^{i_n})\mid \alpha_{i_1,\ldots,i_n}\ne 0\}
$$
and
$$
\omega(h)=\min\{\omega(y_1^{j_1}\cdots y_n^{j_n})\mid \beta_{j_1,\ldots,j_n}\ne 0\}.
$$  
Let $U$ be the ideal $U=(y_1^{i_1}\cdots y_n^{i_n}\mid \alpha_{i_1,\ldots,i_n}\ne 0)$ in $\widehat{S_0}$ and $V$ be the ideal $V=(y_1^{j_1}\cdots y_n^{j_n}\mid \beta_{j_1,\ldots,j_n}\ne 0)$ in $\widehat{S_0}$. Since $\widehat{S_0}$ is a noetherian ring, there exist monomials $M_1,\ldots,M_s, N_1,\ldots,N_t$  in $y_1,\ldots,y_n$ such that $U=(M_1,\ldots,M_s)$ and $V=(N_1,\ldots,N_t)$.  We can further assume that $\omega(M_1)<\omega(M_i)$ for $i>1$ and $\omega(N_1)<\omega(N_j)$ for $j>1$. Since $\omega(f)\ge 0$ we have that $\omega(N_1)\le \omega(M_1)$. There exist units $\alpha_i$ and $\beta_i$ in $S_0$ such that  $\overline M_1=\alpha_1M_1^e,\ldots,\overline M_s=\alpha_sM_s^e,\overline N_1=\beta_1N_1^e,\ldots,\overline N_t=\beta_tN_t^e\in R_0$ are monomials in 
$x_1,\ldots,x_n$. By Proposition \ref{PropZa}, there exists a sequence of PMTs along $\nu$ $R_0\rightarrow R_1$ such that $\overline M_1$ divides $\overline M_i$ in $R_1$ for all $i$, $\overline N_1$ divides $\overline N_j$ for all $j$ in $R_1$ and $\overline M_1$ divides $\overline N_1$ in $R_1$. 
Let $S_0\rightarrow S_1$ be the sequence of PMTs along $\omega$ obtained from applying Proposition \ref{PropSR} to the sequence $R_0\rightarrow R_1$, so that $S_1=R_1[z_1,\ldots,z_m]_\omega$, and $R_1$ has regular parameters $x_1(1),\ldots,x_n(1)$, $S_1$ has regular parameters $y_1(1),\ldots,y_n(1)$  such that there is a unit $\gamma\in S_1$ such that 
$$
x_1(1)=\gamma y_1(1)^e\mbox{ and }x_i(1)=y_i(1)\mbox{ for }2\le i\le n.
$$
Now each of $\overline M_1,\ldots,\overline M_s,\overline N_1,\ldots,\overline N_t$ is a monomial in $y_1(1),\ldots,y_n(1)$ times a unit in $S_1$, so $M_1^e$ divides $M_i^e$ for all $i$ in $S_1$, 
$N_1^e$ divides $N_j^e$ for all $j$ in $S_1$ and $N_1^e$ divides $M_1^e$ in $S_1$. 
Thus $M_1$ divides $M_i$ for all $i$ in $S_1$, 
$N_1$ divides $N_j$ for all $j$ in $S_1$ and $N_1$ divides $M_1$ in $S_1$. 
Since $S_1/\mathfrak m_{S_1}=S_0/\mathfrak m_{S_0}$, we have that $\kappa$ is a coefficient field for $S_1$, and $\widehat S_1=\kappa[[y_1(1),\ldots,y_n(1)]]$. Further, the induced homomorphism
$\widehat{S_0}\rightarrow \widehat{S_1}$ is the natural $\kappa$-algebra homomorphism 
$$
\kappa[[y_1,\ldots,y_n]]\rightarrow \kappa[[y_1(1),\ldots,y_n(1)]]
$$
defined by substitution of the $y_i$ by suitable monomials in $y_1(1),\ldots,y_n(1)$.
Write $g=u_1M_1+u_2M_2+\cdots+u_sM_s$ and $h=v_1N_1+v_2N_2+\cdots+v_tN_t$, with $u_1,\ldots,u_s,v_1,\ldots,v_t\in \widehat{S_0}$ and $u_1,v_1$ units. Now 
$$
g_1=u_1\frac{M_1}{N_1}+u_2\frac{M_2}{N_1}+\cdots+u_s\frac{M_s}{N_1} \in \widehat{S_1}
$$
and
$$
h_1=v_1+v_2\frac{N_2}{N_1}+\cdots+v_t\frac{N_t}{N_1} \in \widehat{S_1}.
$$
so $\frac{g}{h}=g_1h_1^{-1}\in \widehat{S_1}$. Thus by \cite[Lemma 2]{Ab1},
$$
f=\frac{g}{h}\in \widehat{S_1}\cap L=S_1=R_1[z_1,\ldots,z_m]_{\omega}\subset \VR_\nu[z_1,\ldots,z_m]_\omega.
$$

\ \\

\noindent{\footnotesize STEVEN DALE CUTKOSKY\\
 Department of Mathematics,
University of Missouri\\
 Columbia, MO 65211, USA\\
Email:{\tt cutkoskys@missouri.edu}\\\\

\noindent{\footnotesize JOSNEI NOVACOSKI\\
Departamento de Matem\'atica--UFSCar\\
Rodovia Washington Lu\'is, 235\\
13565-905, S\~ao Carlos - SP, Brazil.\\
Email: {\tt josnei@dm.ufscar.br} \\\\

\end{document}